\documentclass[11pt]{amsart}

\newlength{\myhmargin}  
\usepackage[textheight=574pt, textwidth=445pt, marginparwidth=50pt, centering]{geometry}

\usepackage{amsmath,amssymb,amsthm,amsfonts,amscd,flafter,
graphicx,verbatim,pinlabel,mathrsfs}
\usepackage[all]{xy}
\usepackage{epstopdf}
\usepackage[colorlinks=false]{hyperref}
\usepackage[all]{hypcap}
\AtBeginDocument{\addtocontents{toc}{\protect\setlength{\parskip}{0pt}}}

\title{Instanton Floer homology and contact structures}

\author[John A. Baldwin]{John A. Baldwin}
\address{Department of Mathematics \\ Boston College}
\email{john.baldwin@bc.edu}

\author[Steven Sivek]{Steven Sivek}
\address{Department of Mathematics \\ Princeton University}
\email{ssivek@math.princeton.edu}

\thanks{JAB was partially supported by NSF grant DMS-1104688.  SS was partially supported by NSF postdoctoral fellowship DMS-1204387.}

\def\C{{\mathbb{C}}}

\newcommand\hf{\widehat{HF}}

\newcommand\ssm{\smallsetminus}

\newcommand\Psit{\underline{\Psi}}
\newcommand\Z{\mathbb{Z}}

\newcommand\inr{{\rm int}}

\newcommand\data{\mathscr{D}}

\newcommand\SHM{SHM}
\newcommand\SHMt{\underline{\SHM}}

\newcommand\SHI{SHI}
\newcommand\SHIt{\underline{\SHI}}
\newcommand\KHI{KHI}
\newcommand\KHIt{\underline{\KHI}}
\newcommand{\definefunctor}[1]{\textbf{\textup{#1}}}

\newcommand\SHMtfun{\definefunctor{\SHMt}}

\newcommand\SHItfun{\definefunctor{\SHIt}}

\newcommand\Img{{\rm Im}}

\newcommand\PSys{\textbf{\textup{PSys}}}
\newcommand{\RPSys}[1][\RR]{{#1}\mbox{-}\PSys}
\newcommand\DiffSut{\textup{\textbf{DiffSut}}}
\newcommand\CobSut{\textup{\textbf{CobSut}}}

\newcommand{\longcomment}[2]{#2}

\DeclareFontFamily{U}{mathx}{\hyphenchar\font45}
\DeclareFontShape{U}{mathx}{m}{n}{
      <5> <6> <7> <8> <9> <10>
      <10.95> <12> <14.4> <17.28> <20.74> <24.88>
      mathx10
      }{}
\DeclareSymbolFont{mathx}{U}{mathx}{m}{n}
\DeclareFontSubstitution{U}{mathx}{m}{n}
\DeclareMathAccent{\widecheck}{0}{mathx}{"71}
\newcommand{\HMto}{\widecheck{\mathit{HM}}}

\longcomment{
    \RequirePackage{rotating}                   
    \def\HMto{%
       \setbox0=\hbox{$\widehat{\mathit{HM}}$}
       \setbox1=\hbox{$\mathit{HM}$}
       \dimen0=1.1\ht0
       \advance\dimen0 by 1.17\ht1
       \smash{\mskip2mu\raise\dimen0\rlap{%
          \begin{turn}{180}
              {$\widehat{\phantom{\mathit{HM}}}$}
           \end{turn}} \mskip-2mu    
                \mathit{HM}
    }{\vphantom{\widehat{\mathit{HM}}}}{}}
}
    
    \newcommand*\oline[1]{%
  \vbox{%
    \hrule height 0.35pt
    \kern0.1ex
    \hbox{%
      \kern-0.0em
      \ifmmode#1\else\ensuremath{#1}\fi
      \kern-0.1em
    }
  }
}

\newtheorem{theorem}{Theorem}[section]
\newtheorem{lemma}[theorem]{Lemma}

\newtheorem{conjecture}[theorem]{Conjecture}
\newtheorem{corollary}[theorem]{Corollary}
\newtheorem{proposition}[theorem]{Proposition}

\theoremstyle{definition}
\newtheorem{definition}[theorem]{Definition}

\newtheorem{remark}[theorem]{Remark}

\makeatletter
\newtheorem*{rep@thm}{\rep@title}
\newcommand{\newreptheorem}[2]{%
\newenvironment{rep#1}[1][0,0]{%
\def\rep@title{#2##1}%
\begin{rep@thm}}%
{\end{rep@thm}}}
\makeatother
\newreptheorem{theorem}{}

\begin{document}
\begin{abstract} 
We define an invariant of contact 3-manifolds with  convex boundary using Kronheimer and Mrowka's sutured instanton Floer  homology  theory.  To the best of our knowledge, this is the first invariant of contact manifolds---with or without boundary---defined in the instanton Floer setting. We prove that our invariant vanishes for overtwisted contact structures and is nonzero for contact manifolds with boundary which embed into   Stein fillable contact manifolds.  Moreover, we propose a strategy by which  our   contact invariant might be used to relate the fundamental group of a closed contact 3-manifold to properties of its Stein fillings. Our construction is  inspired by a reformulation of a similar invariant in the monopole Floer setting defined by the authors in \cite{bsSHM}.
 \end{abstract}

\maketitle

\section{Introduction}
\label{sec:intro}
Floer-theoretic invariants of contact  manifolds have been responsible for many important results in low-dimensional topology. Notable examples include the invariants of closed contact 3-manifolds defined by Kronheimer and Mrowka \cite{km} and  by Ozsv{\'a}th and Szab{\'o} \cite{osz1} in monopole and Heegaard Floer homology, respectively. Also important is the work in \cite{hkm4}, where Honda, Kazez, and Mati{\'c} extend Ozsv{\'a}th and Szab{\'o}'s construction, using sutured Heegaard Floer homology to define an invariant of \emph{sutured contact manifolds}, which are triples of the form $(M,\Gamma,\xi)$ where $(M,\xi)$ is a contact 3-manifold with convex boundary and  $\Gamma\subset \partial M$ is a multicurve dividing the characteristic foliation of $\xi$ on $\partial M$. Recently, we defined an analogous invariant of sutured contact manifolds in Kronheimer and Mrowka's sutured monopole Floer homology theory \cite{bsSHM}.

The goal of this paper is to define an invariant of sutured contact manifolds in Kronheimer and Mrowka's sutured instanton Floer homology ($SHI$). To the best of our knowledge, this is the first invariant of contact manifolds---with or without boundary---defined in the instanton Floer setting. Like the Heegaard Floer invariants but in contrast with the monopole invariants, our instanton Floer contact invariant is defined using the full relative Giroux correspondence. Its construction is inspired by a reformulation of the monopole Floer invariant in \cite{bsSHM} which was used there to prove that the monopole invariant is well-defined.

A unique feature of the instanton Floer viewpoint is the central role played by  the fundamental group. Along these lines, we conjecture a means by which our  contact invariant in $SHI$ might be used to relate the fundamental group of a closed contact 3-manifold to properties of its Stein fillings, a relationship which has been largely unexplored to this point.

Below, we sketch the construction of our contact invariant, describe some of its most important properties, state some conjectures, and discuss plans for future work which include using the constructions in this paper to define invariants of bordered manifolds in the instanton Floer setting.

\subsection{A contact invariant in $\SHI$} 
\label{ssec:introshi} 
Suppose $(M,\Gamma)$ is a balanced sutured manifold. Roughly speaking, a \emph{closure} of $(M,\Gamma)$ is formed by gluing on some auxiliary piece and ``closing up" by  identifying the remaining boundary components. In \cite{km4}, Kronheimer and Mrowka defined an invariant of balanced sutured manifolds in terms of the instanton Floer homology groups of these related closed 3-manifolds. They proved that the groups associated to different closures of a given sutured manifold are all isomorphic. In this way, their invariant assigns to $(M,\Gamma)$ an isomorphism class of $\C$-modules, denoted by $\SHIt(M,\Gamma)$. 

In \cite{bs3}, we introduced a refinement of their construction which assigns to $(M,\Gamma)$ a \emph{projectively transitive system of $\C$-modules}, denoted by $\SHItfun(M,\Gamma)$. This system records the collection of $\C$-modules---all isomorphic to $\SHIt(M,\Gamma)$---associated to different closures of $(M,\Gamma)$ together with \emph{canonical} isomorphisms relating these modules, where these isomorphisms are  well-defined up to multiplication in $\C^\times$. We refer to this system as the \emph{sutured instanton homology} of $(M,\Gamma)$.

A key step in constructing our contact invariant is to first define maps on sutured instanton homology associated to contact handle attachments. That is, suppose $(M_i,\Gamma_i)$ is a balanced sutured manifold obtained by attaching a contact $i$-handle to $(M,\Gamma)$ for some $i\in\{0,1,2,3\}$. We define a map \[\mathscr{H}_i:\SHItfun(-M,-\Gamma)\to\SHItfun(-M_i,-\Gamma_i)\] which depends only on the smooth data involved in this handle attachment. Our construction of these maps is almost identical to that of the analogous maps in  sutured monopole  homology \cite{bsSHM}; in particular, these maps are defined in terms of the maps on instanton Floer homology induced by natural cobordisms between closures.

Suppose now that $(M,\Gamma,\xi)$ is a sutured contact manifold. According to the relative Giroux correspondence, this contact manifold  admits a  partial open book decomposition. This implies that $(M,\Gamma,\xi)$ can be obtained by attaching contact 2-handles to a sutured contact manifold $H(S)$ formed from rounding the corners of a  tight, vertically invariant contact structure on $S\times I$, where $S$ is a compact surface with boundary (the surface $S$ and the contact $2$-handle attachments are specified by the partial open book decomposition). 
Let \[\mathscr{H}:\SHItfun(-H(S))\to\SHItfun(-M,-\Gamma)\] be the composition of the maps associated to the contact $2$-handle attachments above. Since $H(S)$ is a product sutured manifold, its sutured instanton homology has rank one with generator $\mathbf{1}\in\SHItfun(-H(S))\cong \C,$ and we define the contact invariant of $(M,\Gamma,\xi)$ to be  \[\theta(M,\Gamma,\xi):=\mathscr{H}(\mathbf{1})\in\SHItfun(-M,-\Gamma).\] Our main theorem (stated later as Theorem \ref{thm:well-defined4}) is the following.

\begin{theorem}
The element $\theta(M,\Gamma,\xi)$ is well-defined.
\end{theorem}

That is to say, this element does  not depend on the chosen partial open book decomposition  (by the Giroux correspondence, it suffices to prove that this element is preserved under positive stabilization of the open book).\footnote{The analogous contact invariant in sutured monopole  homology is defined without reference to the relative Giroux correspondence, but admits the same formulation as above.} 

We also show that this contact invariant behaves naturally with respect to the contact handle attachment maps, per the following (stated later as Theorem \ref{thm:handletheta2}).

\begin{theorem}
\label{thm:handletheta} Suppose  $(M_i,\Gamma_i,\xi_i)$  is obtained from $(M,\Gamma,\xi)$ by attaching a contact $i$-handle and $\mathscr{H}_i$ is the associated map for $i=0,$ $1,$ or $2$. Then $\mathscr{H}_i(\theta(M,\Gamma,\xi))=\theta(M_i,\Gamma_i,\xi_i).$\footnote{We believe the contact $3$-handle maps also preserve the contact invariant but do not prove this here.}
\end{theorem}

The invariant $\theta$ shares several important features with Honda, Kazez, and Mati{\'c}'s invariant and with our contact invariant in sutured monopole homology (besides the one above). Among these are the following two results (stated later as Theorems \ref{thm:otinstanton} and \ref{thm:stein}). Interestingly, the proofs of both theorems below are substantially different from those of their counterparts in \cite{bsSHM} in the sutured monopole homology setting.

\begin{theorem}
\label{thm:introot}
If $(M,\Gamma,\xi)$ is overtwisted, then $\theta(M,\Gamma,\xi)=0$.
\end{theorem}

For the next theorem, suppose $(Y,\xi)$ is a closed contact 3-manifold and let $Y(n)$ denote the sutured manifold obtained by removing $n$ disjoint Darboux balls for any $n\geq 1$.

\begin{theorem}
\label{thm:introstein}
If $(Y,\xi)$ is Stein fillable, then $\theta(Y(n),\xi|_{Y(n)})\neq 0$.
\end{theorem}

As we shall see, the corollary below (stated later as Corollary \ref{cor:embedding2}) follows   from Theorems \ref{thm:introstein} and \ref{thm:handletheta}.
\begin{corollary}
\label{cor:embedding}
If $(M,\Gamma,\xi)$ embeds into a Stein fillable manifold, then $\theta(M,\Gamma,\xi)\neq 0$.
\end{corollary}

In a related direction, we conjecture the following, which is an instanton Floer analogue of a theorem of Plamenevskaya  regarding the contact invariant in Heegaard Floer homology \cite{pla3}.  

\begin{conjecture}
\label{conj:steinrank} Suppose $J_1,J_2$ are  Stein structures on a smooth $4$-manifold $X$ such that $c_1(J_1)- c_1(J_2)$ is nontorsion. Let $\xi_1,\xi_2$ be the induced contact structures on $Y=\partial X$. Then the contact invariants \[\theta(Y(1),\xi_1|_{Y(1)})\,\,\,\,\,{\rm and}\,\,\,\,\,\theta(Y(1),\xi_2|_{Y(1)})\]  are linearly independent in $\SHItfun(-Y(1))$.
\end{conjecture}

Note that one needs some kind of naturality for a statement like that in Conjecture \ref{conj:steinrank} since ``linear independence" has little meaning if elements are only well-defined up to isomorphism. This  is precisely the sort of consideration that motivated our work in \cite{bs3}.  

As we explain in Section \ref{sec:stein}, a positive answer to Conjecture \ref{conj:steinrank} would imply the following link between the fundamental group of a contact 3-manifold and properties of its Stein fillings.

\begin{conjecture}
\label{conj:steinpi1} Suppose $Y$ is  an integer homology 3-sphere which bounds a Stein 4-manifold $(X,J)$ with $c_1(J)\neq 0$. Then there exists a nontrivial  homomorphism $\rho:\pi_1(Y)\to SU(2).$
\end{conjecture}

It was pointed out to us by Tom Mrowka that the conclusion of Conjecture \ref{conj:steinpi1} holds by arguments similar to those used in the proof of the Property P conjecture \cite{km2} if the Stein filling has $b_2^+>0$. However, this leaves a lot of Stein fillable contact structures behind. For instance, Etnyre \cite{et} shows  that if a contact structure  is supported by a planar open book, then all of its Stein fillings are negative definite.

In light of Conjecture \ref{conj:steinpi1}, it is  natural to ask whether there exist \emph{any} integer homology spheres other than $S^3$ whose fundamental group admits no nontrivial $SU(2)$ representations? The main result of \cite{km7} implies that the answer is ``no" among integer homology spheres arising from surgery on knots in $S^3$. In general, however, the question seems to be  wide open.



Finally, it is worth mentioning that Conjecture \ref{conj:steinpi1} would also follow from Plamenevskaya's work in \cite{pla3}, combined with the conjectural isomorphism between $\SHIt(Y(1))$ and $\hf(Y)\otimes\C$ proposed in \cite{km4}, but the latter seems more difficult to establish than Conjecture \ref{conj:steinrank}.

\subsection{Future  directions} 
\label{ssec:future} Two of our future projects involve defining sutured cobordism maps and bordered invariants in the instanton Floer setting as mentioned briefly below.

Suppose  $(M,\Gamma)$ is a sutured submanifold of $(M',\Gamma')$ and  $\xi$ is a contact structure on $M'\ssm \inr(M)$ with  dividing set $\Gamma\cup \Gamma'$. Note that $(M'\ssm \inr(M),\Gamma \cup \Gamma',\xi)$ can be obtained from a vertically invariant contact structure on $\partial M\times I$ by attaching contact handles. Given such a handle decomposition $H$, we may then define  \[\Phi_{\xi,H}:\SHItfun(-M,-\Gamma)\to\SHItfun(-M',-\Gamma')\] to be the corresponding composition of contact handle attachment maps. 
A similar map was defined by Honda, Kazez, and Mati{\'c} in \cite{hkm5} in the setting of sutured Heegaard Floer homology (see also \cite{bsSHM} in the setting of sutured monopole homology). Their map depends only on $\xi$ and we conjecture that the same is true for the map above.

\begin{conjecture}
\label{conj:H} The map $\Phi_{\xi,H}$ is independent of $H$.
\end{conjecture}

A positive answer to this conjecture would allow us assign well-defined maps  to cobordisms between sutured manifolds in the instanton Floer setting---in the language of \cite{bs3}, to extend $\SHItfun$ to a functor from $\CobSut$  to $\RPSys[\C]$---following Juhasz's strategy \cite{juhasz3}, as explained in the analogous context of sutured monopole  in \cite[Subsection 1.3]{bsSHM}. And this, in turn, would allow us to define invariants of bordered 3-manifolds using instanton Floer homology, following a strategy of Zarev \cite{zarev}; again, see  \cite[Subsection 1.3]{bsSHM} for the analogous discussion in the sutured monopole Floer setting.

\subsection{Organization} 
In Section \ref{sec:prelims}, we provide the necessary background on projectively transitive systems, sutured instanton homology, and the relative Giroux correspondence. In Section \ref{sec:handlemaps}, we define the contact handle attachment maps mentioned above. In Section \ref{sec:instantoninvt}, we define the contact invariant $\theta$ and establish some basic properties of this invariant, proving Theorems \ref{thm:handletheta}, \ref{thm:introot}, \ref{thm:introstein}, and Corollary \ref{cor:embedding}. Finally, in Section \ref{sec:stein}, we explain how a positive answer to Conjecture \ref{conj:steinrank} would imply a positive answer to  Conjecture \ref{conj:steinpi1}.
\subsection{Acknowledgements} We thank Peter Kronheimer, Tye Lidman, Tom Mrowka, and Vera V{\'e}rtesi for helpful conversations.

\section{Preliminaries}
\label{sec:prelims}
In this section, we review the notion of a projectively transitive system, the construction of sutured instanton homology, and the relative Giroux correspondence.

\subsection{Projectively transitive systems of $\C$-modules}

\label{ssec:transitive-systems}
In \cite{bs3} we introduced \emph{projectively transitive systems} to make precise the idea of a collection of modules being canonically isomorphic up to multiplication by a unit. We recount their definition and related notions below, focusing on modules over $\C$.

\begin{definition}
Suppose $M_\alpha$ and $M_\beta$ are $\C$-modules. We say that  elements $x,y\in M_\alpha$ are \emph{equivalent} if $x=u\cdot y$ for some  $u\in \C^\times$. Likewise,  homomorphisms \[f,g: M_\alpha \to M_\beta\]  are  \emph{equivalent} if $f=u\cdot g$ for some  $u\in \C^\times$.  
\end{definition}

\begin{remark}
We will write $x\doteq y$ or $f\doteq g$ to indicate that two  elements or homomorphisms are equivalent, and  will denote their equivalence classes by $[x]$ or $[f]$.
\end{remark}

Note that \emph{composition} of  equivalence classes of homomorphisms is well-defined, as is the \emph{image} of the equivalence class of an element under an equivalence class of homomorphisms.

\begin{definition}
 A \emph{projectively transitive system of $\C$-modules} consists of a set $A$ and:
\begin{enumerate}
\item a collection of $\C$-modules $\{M_\alpha\}_{\alpha \in A}$ and
\item a collection of equivalence classes of homomorphisms $\{g^\alpha_\beta\}_{\alpha,\beta \in A}$ such that:
\begin{enumerate}
\item elements of the equivalence class $g^\alpha_\beta$ are isomorphisms from $M_\alpha$ to $M_\beta$, 
\item  $g^\alpha_\alpha=[id_{M_\alpha}]$,
\item $g^\alpha_\gamma = g^\beta_\gamma \circ g^\alpha_\beta$.
\end{enumerate}
\end{enumerate}
\end{definition}


\begin{remark}The equivalence classes  of homomorphisms in a projectively transitive system of $\C$-modules can  be thought of as specifying canonical isomorphisms between the modules in the system that are well-defined up to multiplication by units in $\C$. 
\end{remark}


The class of  projectively transitive systems of $\C$-modules forms a category $\RPSys[\C]$ with the following notion of morphism.

\begin{definition}
\label{def:projtransysmor} A \emph{morphism} of  projectively transitive systems of $\C$-modules \[F:(A,\{M_{\alpha}\},\{g^{\alpha}_{\beta}\})\to(B,\{N_{\gamma}\},\{h^{\gamma}_{\delta}\})\]  is  a collection of equivalence classes  of homomorphisms $F=\{F^\alpha_\gamma\}_{\alpha\in A,\,\gamma\in B}$  such that:
\begin{enumerate}
\item elements of the equivalence class $F^\alpha_\gamma$ are homomorphisms from $M_\alpha$ to $N_\gamma$,
\item \label{eqn:projtransysmor} $F^\beta_\delta\circ g^\alpha_\beta = h^\gamma_\delta\circ F^\alpha_\gamma$.
\end{enumerate}  
Note that $F$ is an \emph{isomorphism} iff the elements in each equivalence class $F^\alpha_\gamma$ are isomorphisms.
\end{definition}

\begin{remark}\label{rmk:completesubset} A collection of equivalence classes of homomorphisms $\{F^\alpha_\gamma\}$  with indices  ranging over any nonempty subset of $A\times B$ can be uniquely completed to a morphism as long as this collection satisfies the compatibility in \eqref{eqn:projtransysmor} where it makes sense.
\end{remark}

\begin{remark}
\label{rmk:systemofsystems}
Suppose $\{\mathcal{S}_\alpha\}_{\alpha\in A}$ is a collection of projectively transitive systems of $\C$-modules and \[\{f_{\alpha,\beta}:\mathcal{S}_{\alpha}\to\mathcal{S}_\beta\}_{\alpha,\beta\in A}\] is a collection of isomorphisms of projectively transitive systems of $\C$-modules which satisfy the transitivity $f_{\alpha,\gamma}=f_{\beta,\gamma}\circ f_{\alpha,\beta}$ for all $\alpha,\beta,\gamma\in A$. Then this \emph{transitive system of systems} defines an even larger projectively transitive system of $\C$-modules in a natural way, whose set of constituent $\C$-modules is the union over all $\alpha\in A$ of the sets of $\C$-modules making up the systems $\mathcal{S}_\alpha$.
\end{remark}

\begin{definition}
\label{def:element}
 An \emph{element} of a projectively transitive system of $\C$-modules \[x\in\mathcal{M}=(A,\{M_{\alpha}\},\{g^{\alpha}_{\beta}\})\]
 is a collection of equivalence classes of elements $x = \{x_{\alpha}\}_{\alpha\in A}$ such that:
\begin{enumerate}
\item elements of the equivalence class $x_{\alpha}$ are elements of $M_\alpha$,
\item \label{eqn:projtransysmorelt} $x_\beta = g^\alpha_\beta(x_\alpha)$.
\end{enumerate}  
\end{definition}

\begin{remark}
\label{rmk:completesubset2}
 As in Remark \ref{rmk:completesubset}, a collection of equivalence classes of elements $\{x_{\alpha}\}$ with indices ranging over any nonempty subset of $A$ can be uniquely completed to an element of $\mathcal{M}$ as long as this collection satisfies the compatibility in \eqref{eqn:projtransysmorelt} where it makes sense.
 \end{remark}
 
We  say that $x$ is a \emph{generator} in $\mathcal{M}$  if each $M_\alpha$ is isomorphic to $\C$ and each $x_{\alpha}$ is the equivalence class of a generator---i.e., nonzero. 
The \emph{zero} element $0\in \mathcal{M}$ is the collection of equivalence classes of the elements $0\in M_\alpha$. Finally, it is clear how to define the image $F(x)$ of an element $x\in \mathcal{M}$ under a morphism $F:\mathcal{M}\to\mathcal{N}$ of projectively transitive systems of $\C$-modules.

\begin{remark} Given a $\C$-module $M$, we can also think of $M$ as the projectively transitive system of $\C$-modules given (in an abuse of notation) by  \[ M=(\{\star\},\{M\}, \{[id_M]\})\] consisting of the single $\C$-module $M$ together with the equivalence class of the identity map, so that it makes sense to write $\mathcal{S}\cong M$, for any other object $\mathcal{S}\in \RPSys[\C]$.
\end{remark}


\subsection{Sutured instanton homology} 
\label{ssec:shi} In this subsection, we describe our refinement in \cite{bs3} of Kronheimer and Mrowka's sutured instanton homology, as defined in \cite{km4}.

\subsubsection{Closures of balanced sutured manifolds}
\begin{definition} 
\label{def:sutured}A \emph{balanced sutured manifold}  $(M,\Gamma)$ is a compact, oriented, smooth 3-manifold $M$ with a collection $\Gamma$ of disjoint, oriented,  smooth curves in $\partial M$ called \emph{sutures}. Let $R(\Gamma) = \partial M\smallsetminus\Gamma$, oriented as a subsurface of $\partial M$. We require that:
\begin{enumerate}
\item neither $M$ nor $R(\Gamma)$ has  closed components,
\item $R(\Gamma) = R_+(\Gamma)\sqcup R_-(\Gamma)$ with $\partial R_+(\Gamma) = -\partial R_-(\Gamma) = \Gamma$,
\item $\chi(R_+(\Gamma)) = \chi(R_-(\Gamma))$.
\end{enumerate}
 \end{definition}

An \emph{auxiliary surface} for $(M,\Gamma)$ is a compact, connected, oriented surface $F$ with $g(F)>0$ and $\pi_0(\partial F)\cong \pi_0(\Gamma)$.  Suppose $F$ is an auxiliary surface for $(M,\Gamma)$, $A(\Gamma)$ is a closed tubular neighborhood of $\Gamma$ in $\partial M$, and    \[h:\partial F\times[-1,1]\rightarrow A(\Gamma)\] is an orientation-reversing diffeomorphism which sends $\partial F\times \{\pm 1\}$ to $\partial (R_{\pm}(\Gamma)\smallsetminus A(\Gamma)).$ One forms a   \emph{preclosure} of $M$ \begin{equation*}\label{eqn:bF}M'=M\cup_h F\times [-1,1]\end{equation*}  by gluing $F\times[-1,1]$ to $M$ according to $h$ and rounding corners.  This preclosure has two diffeomorphic boundary components, $\partial_+ M'$ and $\partial_- M'$. We may therefore glue $\partial_+M'$ to $\partial_-M'$ by some diffeomorphism   to form a closed manifold $Y$ containing a distinguished surface \[R:=\partial_+M' = -\partial_- M'\subset Y.\] In \cite{km4}, Kronheimer and Mrowka define a \emph{closure} of $(M,\Gamma)$ to be any pair $(Y,R)$ obtained in this way.  Our definition of closure, as needed for naturality, is  slightly more involved.

\begin{definition}[\cite{bs3}]
\label{def:smoothclosure} A \emph{marked odd closure} of $(M,\Gamma)$ is a tuple $\data = (Y,R,r,m,\eta,\alpha)$ consisting of:
\begin{enumerate}
\item a closed, oriented,  3-manifold $Y$,
\item  a closed, oriented,  surface $R$ with $g(R)\geq 2$,
\item an oriented, nonseparating, embedded curve $\eta\subset R$,
\item a smooth, orientation-preserving embedding $r:R\times[-1,1]\hookrightarrow Y$,
\item a smooth, orientation-preserving embedding $m:M\hookrightarrow Y\smallsetminus\inr(\Img(r))$ such that: 
\begin{enumerate}
\item $m$ extends  to a diffeomorphism \[M\cup_h F\times [-1,1]\rightarrow Y\smallsetminus{\rm int}(\Img(r))\] for some $A(\Gamma)$, $F$, $h$, as above,
\item $m$ restricts to an orientation-preserving embedding \[R_+(\Gamma)\smallsetminus A(\Gamma)\hookrightarrow r(R\times\{-1\}).\]
\end{enumerate}
\item an oriented, embedded curve $\alpha\subset Y$ such that:
\begin{enumerate}
\item $\alpha$ is disjoint from $\Img(m)$,
\item $\alpha$ intersects $r(R\times[-1,1])$ in an arc of the form $r(\{p\}\times[-1,1])$ for some  $p\in R$.
\end{enumerate}
 \end{enumerate} 
 The \emph{genus} $g(\data)$ refers to the genus of $R$.
\end{definition}


  \begin{remark} Suppose $\data = (Y,R,r,m,\eta,\alpha)$ is a marked odd closure of $(M,\Gamma)$. Then, the tuple \[-\data:=(-Y,-R,r,m,-\eta,-\alpha),\] obtained  by reversing the orientations of $Y$, $R$, $\eta$, and $\alpha$, is a marked odd closure of $-(M,\Gamma):=(-M,-\Gamma),$ where, $r$ and $m$ are the induced embeddings  of $-R\times[-1,1]$ and $-M$ into $-Y$. 
  \end{remark}

\subsubsection{Instanton Floer homology}
Before defining sutured instanton homology, we recall the basic set up of instanton Floer homology from \cite{km4}.

Suppose $Y$ is a closed, oriented, smooth 3-manifold and $w \to Y$ is a Hermitian line bundle  such that  $c_1(w)$ has odd pairing with some  class in $H_2(Y;\Z)$.  Let $E \to Y$ be a $U(2)$ bundle with an isomorphism $\theta:\Lambda^2 E \to w$. Let $\mathcal{C}$ be the space of $SO(3)$ connections on $\operatorname{ad}(E)$ and let $\mathcal{G}$ be the group of determinant-1 gauge transformations of $E$ (the automorphisms of $E$ that respect $\theta$).  The associated instanton Floer homology group, which Kronheimer and Mrowka denote by  $I_*(Y)_w$, is  the  $\Z/8\Z$-graded $\C$-module arising from the Morse homology of the Chern-Simons functional on $\mathcal{C}/\mathcal{G}$ (cf.  \cite{donaldson}).  Given any closed, embedded surface $R\subset Y$ there is a natural operator
\[ \mu(R): I_*(Y)_w \to I_{*}(Y)_w \]
of degree $-2$. When $R$ has genus at least 2, Kronheimer and Mrowka define the submodule
\[ I_*(Y|R)_w \subset I_*(Y)_w \]
to be the eigenspace of $\mu(R)$ with eigenvalue $2g(R)-2$.

Suppose $\alpha$ is an oriented, smooth 1-cycle in $Y$ which intersects a closed, embedded  surface in an odd number of points. One can associate to $(Y,\alpha)$ an instanton Floer  group after first choosing  bundles $w$, $E$, and an isomorphism $\theta$ as above, where  the first Chern class is Poincar{\'e} dual to $\alpha$. This Floer group is itself not an invariant of $(Y,\alpha)$ as it depends on these auxiliary choices. However, given a pair $(Y,\alpha)$, the Floer groups associated to any two sets of auxiliary choices are related by a canonical isomorphism which is well-defined up to sign (cf. \cite[Section 4]{km3}). In particular, the pair $(Y,\alpha)$ defines a projectively transitive system of $\C$-modules, which we will denote by $I_*(Y)_{\alpha}$. The canonical isomorphisms respect the eigenspace decompositions and, so, for a closed embedded surface $R\subset Y$, we may also define the projectively transitive system of $\C$-modules $I_*(Y|R)_{\alpha}$. 

Suppose  $R_1$ and $R_2$ are embedded surfaces in $Y_1$ and $Y_2$ as above. A cobordism $(W,\nu)$ from $(Y_1,\alpha_1)$ to $(Y_1,\alpha_2)$ together with an embedded surface $R_W\subset W$ containing $R_1$ and $R_2$ as components gives rise to a map of projectively transitive systems \[I_*(W|R_W)_{\nu}:I_*(Y_1|R_1)_{\alpha_1}\to I_*(Y_2|R_2)_{\alpha_2}.\] This map depends only on the  homology class $[\nu]\subset H_2(W,\partial W;\mathbb{Z})$ and the isomorphism class of $(W,\nu)$, where two pairs  are isomorphic if they are diffeomorphic by a map which intertwines the boundary identifications. 

\subsubsection{Sutured instanton homology}

Following Kronheimer and Mrowka \cite{km4}, we made the following definition in \cite{bs3}.

\begin{definition}
\label{def:shitwisted}
Given a marked odd closure $\data=(Y,R,r,m,\eta,\alpha)$ of $(M,\gamma)$, the \emph{twisted sutured instanton homology of $\data$} is the projectively transitive system of $\C$-modules  \[\SHIt(\data) :=I_*(Y|r(R\times\{0\}))_{\alpha\,\sqcup\, \eta}.\] 
\end{definition}

\begin{remark}
If $w$ and $u$ are line bundles over $Y$ with first Chern classes represented by $\alpha$ and $\eta$, then the line bundle $w\otimes u$ has first Chern class represented by $\alpha \sqcup\eta$.
\end{remark}

In \cite{bs3}, we constructed canonical isomorphisms \[\Psit_{\data,\data'}:\SHIt(\data)\to\SHIt(\data')\]  for any two marked odd closures $\data,\data'$ of $(M,\Gamma)$, so that \[\Psit_{\data,\data''}=\Psit_{\data',\data''}\circ \Psit_{\data,\data'}\] for all  $\data,\data',\data''$. In other words, the  systems in $\{\SHIt(\data)\}_{\data}$ and the maps in $\{\Psit_{\data,\data'}\}_{\data,\data'}$ form a transitive system of systems and, therefore, a larger projectively transitive system of $\C$-modules as explained in Remark \ref{rmk:systemofsystems}. These isomorphisms are defined almost exactly as in the monopole setting---in terms of $2$-handle or splicing cobordisms depending on whether the genera of $\data$ and $\data'$ are the same or different.

\begin{definition}
The \emph{sutured instanton homology of $(M,\Gamma)$} is the projectively transitive system of $\C$-modules $\SHItfun(M,\Gamma)$ defined by the transitive system of systems above.
\end{definition}

Sutured monopole homology is functorial in the following sense. Suppose \[f:(M,\Gamma)\to(M',\Gamma')\] is a diffeomorphism of sutured manifolds and $\data' = (Y',R',r',m',\eta',\alpha')$ is a marked odd closure of $(M',\Gamma')$. Then \begin{equation}\label{eqn:dataf}\data'_f:=(Y',R',r',m'\circ f,\eta',\alpha')\end{equation} is a marked odd closure of $(M,\Gamma)$. Let \[id_{\data'_f,\data'}: \SHIt(\data'_f)\to\SHIt(\data')\] be the identity map on $\SHIt(\data'_f) = \SHIt(\data')$. The equivalence classes underlying these identity maps can be completed to a morphism (as in Remark \ref{rmk:completesubset}) \[\SHItfun(f):\SHItfun(M,\Gamma)\to\SHItfun(M',\Gamma'),\]  which is an invariant of the isotopy class of $f$. We proved in \cite{bs3} that these morphisms behave as expected under composition of diffeomorphisms, so that  $\SHItfun$ defines a functor from $\DiffSut$ to $\RPSys[\C],$ where $\DiffSut$ is the category of balanced sutured manifolds and isotopy classes of diffeomorphisms between them. 

The following will be important in our definition of the instanton Floer contact invariant.

\begin{proposition}
\label{prop:productsuturedinstanton} If $(M,\Gamma)$ is a product sutured manifold, then $\SHItfun(M,\Gamma)\cong\C$.
\end{proposition}

\begin{proof}
Let $F$ be an auxiliary surface for $(M,\Gamma)$ with $g(F)\geq 2$. Thinking of $(M,\Gamma)$ as  obtained from  $(S\times[-1,1],\partial S\times\{0\})$ by rounding corners, we can  form a preclosure of $(M,\Gamma)$ by gluing $F\times[-1,1]$ to $S\times[-1,1]$ according to a map \[h:\partial F\times[-1,1]\to\partial S\times[-1,1]\] of the form $f\times id$ for some diffeomorphism $f:\partial F\to \partial S$. This preclosure is then a product $M'=(S\cup F)\times[-1,1]$. To form a marked odd closure, we take $R=S\cup F$ and  glue $R\times[-1,1]$ to $M'$ by the ``identity" maps 
\[
R\times\{\pm 1\}\to S\times\{\mp 1\}.
\]
An oriented, nonseparating curve $\eta\subset S\cup F$ and a curve $\alpha = \{p\}\times S^1$ for any point $p\in F$ gives a marked odd closure \[\data = ((S \cup F)\times S^1, (S\cup F),r,m,\eta,\alpha).\] Here, we are thinking of $S^1$ as the union of two copies of $[-1,1]$, and $r$ and $m$ as the obvious embeddings. The system $\SHIt(\data)$ is then given by  \[I_*((S \cup F)\times S^1|(S\cup F)\times\{0\})_{\alpha\,\sqcup\,\eta}\cong\C,\] where this isomorphism follows from \cite[Proposition 7.8]{km4}. \end{proof}

\subsection{The relative Giroux correspondence}
\label{ssec:relgiroux}

Below, we review the relative Giroux correspondence between partial open books and sutured contact manifolds. Our discussion of  this correspondence differs slightly in style but not in substance from the discussions in \cite{etguozbagcirelative,hkm4}.

\begin{definition} 
\label{def:POB}A \emph{partial open book} is a quadruple $(S,P,h,\mathbf{c})$, where: 
\begin{enumerate}
\item $S$ is a surface with nonempty boundary, 
\item $P$ is a subsurface of $S$,  
\item $h:P\to S$ is an embedding which restricts to the identity on $\partial P\cap \partial S$, 
\item $\mathbf{c}=\{c_1,\dots,c_n\}$ is a set of disjoint, properly embedded arcs in $P$ such that $S\ssm \mathbf{c}$ deformation retracts onto $S\ssm P$.
\end{enumerate}
\end{definition}

\begin{remark}
The collection $\mathbf{c}$ of \emph{basis arcs} for $P$ is not typically recorded in the data of a partial open book. Usually, it is just required that $S$ be obtained from $\overline{S\ssm P}$ by successive  $1$-handle attachments. The basis arcs specify a $1$-handle decomposition of $P$. Given that we are specifying basis arcs, we  do not technically need to record the subsurface $P$. We do so anyhow to emphasize the equivalence between Definition \ref{def:POB} and the more commonplace definition of partial open book found in \cite{etguozbagcirelative,hkm4}.
\end{remark}

Suppose $(S,P,h,\mathbf{c})$ is a partial open book. Consider the $[-1,1]$-invariant contact structure $\xi_S$ on $S\times[-1,1]$ for which each $S\times\{t\}$ is convex with collared Legendrian boundary and the dividing set on $S\times\{1\}$ consists of $k$ boundary parallel arcs, one for each component of $\partial S$, oriented in the same direction as the boundary, as shown in Figure \ref{fig:productsurfacecorners}. Let $H(S)$ be the \emph{product sutured contact manifold}  obtained from $(S\times[-1,1],\xi_{S})$ by rounding corners. 

\begin{remark}Note that $H(S)$ is precisely the sort of contact handlebody that appears in the  Heegaard splitting associated to an open book with page $S$. 
\end{remark}

Let $\gamma_i$ be the curve on $\partial H(S)$ corresponding to  \begin{equation}\label{eqn:basishandle}(c_i\times\{1\})\cup (\partial c_i\times [-1,1])\cup (h(c_i)\times\{-1\})\,\subset\, \partial (S\times[-1,1]).\end{equation}  Let $M(S,P,h,\mathbf{c})$ be the sutured contact manifold obtained from $H(S)$ by attaching contact $2$-handles along the curves in \begin{equation}\label{eqn:gammac}\boldsymbol{\gamma}(h,\mathbf{c}):=\{\gamma_1,\dots,\gamma_n\}.\end{equation} We will  use $H(S)$ and $M(S,P,h,\mathbf{c})$ to refer both to these sutured contact manifolds and to the sutured manifolds underlying them.

\begin{figure}[ht]
\centering
\includegraphics[width=11cm]{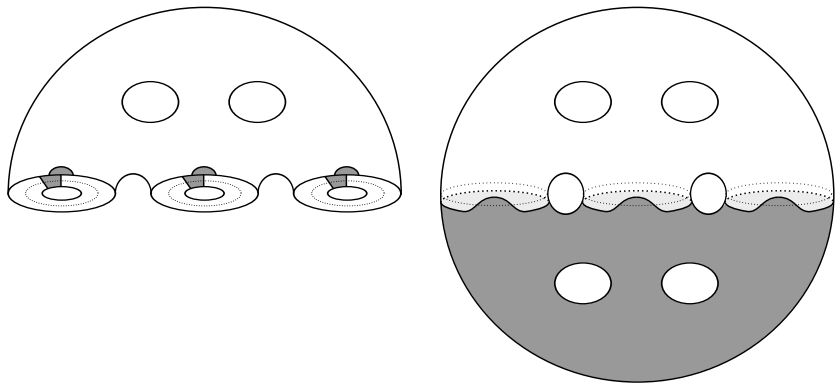}
\caption{Left, $(S\times[-1,1],\xi_{S})$,  with  negative region shaded, for a genus $2$ surface with $3$ boundary components. Right, the convex boundary of the sutured contact handlebody $H(S)$ obtained by rounding corners.}
\label{fig:productsurfacecorners}
\end{figure}


\begin{definition}
A \emph{partial open book decomposition} of $(M,\Gamma,\xi)$ is a partial open book $(S,P,h,\mathbf{c})$ together with a contactomorphism \[f:M(S,P,h,\mathbf{c})\to (M,\Gamma,\xi).\]
\end{definition}

The ``existence" part of the relative Giroux correspondence between partial open books and sutured contact manifolds, proven by Honda, Kazez, and Mati{\'c} in \cite{hkm4}, says the following.

\begin{theorem}
\label{thm:relativegiroux1}
Every sutured contact manifold admits a partial open book decomposition.
\end{theorem}

Below, we describe how  different partial open book decompositions of $(M,\Gamma,\xi)$ are related. Suppose $(S,P,h,\mathbf{c})$ and $(S',P',h',\mathbf{c}')$ are partial open books. Note that a diffeomorphism \begin{equation}\label{eqn:g}g:(S,P,\mathbf{c})\to(S',P',\mathbf{c}')\end{equation} which intertwines $h$ and $h'$   gives rise to a canonical isotopy class of  contactomorphisms \begin{equation}\label{eqn:tildeg}\tilde g:H(S)\to H(S')\end{equation} and therefore to a canonical isotopy class of contactomorphisms \begin{equation}\label{eqn:barg}\bar{\tilde g}:M(S,P,h,\mathbf{c})\to M(S',P',h',\mathbf{c}').\end{equation}



\begin{definition}
\label{def:isomorphic} We say that  $(S,P,h,\mathbf{c},f)$ and $(S',P',h',\mathbf{c}',f')$ are \emph{isomorphic} partial open book decompositions if there exists a diffeomorphism $g$ as in \eqref{eqn:g} such that $f=f'\circ\bar{\tilde g}$.
\end{definition}

\begin{definition}
\label{def:posstab}  A \emph{positive stabilization} of the partial open book  $(S,P,h,\mathbf{c})$ is a partial open book $(S',P',h',\mathbf{c'})$ such that: 
\begin{enumerate}
\item $S'$ is obtained by attaching a $1$-handle $H_0$ to $S$,
\item $P' = P\cup H_0$,
\item $h' = D_{\beta}\circ h,$ where $\beta$ is a curve on $S'$ meeting a cocore $c_0$ of $H_0$ exactly once, and $D_{\beta}$ denotes a positive Dehn twist along $\beta$,
\item $\mathbf{c}' = \mathbf{c}\cup \{c_0\}$.  
\end{enumerate}
\end{definition}

Suppose $(S',P',h',\mathbf{c'})$ is a positive stabilization of $(S,P,h,\mathbf{c})$ as in the definition above. Let $M(S',P',h',c_0)$ be the sutured contact manifold obtained from $H(S')$ by attaching a contact 2-handle along the curve $\gamma_0'\subset \partial H(S')$ obtained from $c_0$ as in \eqref{eqn:basishandle}. Note that $M(S',P',h',c_0)$ is obtained from $H(S)$ by attaching a Darboux ball in the form of cancelling contact $1$- and $2$-handles. In particular, there is a canonical isotopy class of contactomorphisms \begin{equation}\label{eqn:q}q:M(S',P',h',c_0)\to H(S)\end{equation} which restricts to the identity away from this Darboux ball and sends the curves $\gamma_1',\dots,\gamma_n'\subset \partial M(S',P',h',c_0)$ to $\gamma_1,\dots,\gamma_n\subset\partial H(S)$. Such a map gives rise to a canonical isotopy class of contactomorphisms \begin{equation}\label{eqn:tildeq}\bar q:M(S',P',h',\mathbf{c}')\to M(S,P,h,\mathbf{c}).\end{equation}

\begin{definition}
 A \emph{positive stabilization} of the partial open book decomposition $(S,P,h,\mathbf{c},f)$ is a partial open book decomposition $(S',P',h',\mathbf{c}',f' = f\circ\bar q),$ where $(S',P',h',\mathbf{c}')$ is a positive stabilization of $(S,P,h,\mathbf{c})$  and $\bar q$ is the contactomorphism  in \eqref{eqn:tildeq}.
\end{definition}

The ``uniqueness" part  of the relative Giroux correspondence says the following.

\begin{theorem}
\label{thm:relativegiroux2} Given two partial open book decompositions of the same sutured contact manifold,  it is possible to positively stabilize each some  number of times so that the resulting partial open book decompositions are isomorphic.
\end{theorem}

\begin{remark}As stated, Theorem \ref{thm:relativegiroux2} is a combination of the results in \cite{etguozbagcirelative,hkm4}.  Namely, Etg{\"u} and {\"O}zba{\u{g}}c{\i}'s work in \cite{etguozbagcirelative} implies that a partial open book decomposition of $(M,\Gamma,\xi)$ as defined above determines a contact cell decomposition of $(M,\Gamma,\xi)$. In \cite{hkm4}, Honda, Kazez, and Mati{\'c} prove that  two contact cell decompositions of $(M,\Gamma,\xi)$ admit a common subdivision, and subdividing in their sense corresponds to positive stabilization  as defined above.
\end{remark}

\section{Contact handle attachment maps}
\label{sec:handlemaps}

In this section, we define the contact handle attachment maps in sutured instanton homology mentioned in the introduction. Our construction of these maps is nearly identical to that of the corresponding maps in sutured  monopole homology  \cite{bsSHM}, except that we make comparatively little  reference to contact geometry here. 

\subsection{0-handles}
\label{ssec:0handles}
Attaching a contact $0$-handle to $(M,\Gamma)$ is equivalent to taking the disjoint union of $(M,\Gamma)$ with the  Darboux ball $(B^3,S^1, \xi_{std})$. Let $(M_0,\Gamma_0)$ be this disjoint union. It is not hard to construct a marked odd  closure of $(M_0,\Gamma_0)$ which is also a marked odd closure of $(M,\Gamma)$. We may therefore define the  $0$-handle attachment map to be the ``identity" map. 

Indeed, suppose $M_0'$ is a preclosure of $(M_0,\Gamma_0)$ formed from an auxiliary surface $F_0$. Then there are  natural identifications \[\partial_{\pm}M_0'=R_{\pm}(\Gamma)\cup F_0\cup R_{\pm}(S^1).\]   Let $R$ be a copy of $\partial_+M_0'$. Let $Y_0$ be the closed 3-manifold obtained by gluing $R\times[-1,1]$ to $M_0'$ by the ``identity" map from $R\times\{-1\}$ to $\partial_+M_0'$ and by a map from $R\times\{+1\}$ to $\partial_-M_0'$ which sends a point \[p\in F_0\subset R\,\,\,\,\,\,\text{ to }\,\,\,\,\,\,p\in F_0\subset  \partial_-M_0'.\] Let $\eta\subset R$ be an oriented, nonseparating curve contained in $F_0\subset R$ and let $\alpha\subset Y_0$ be the union of the oriented arcs \[\{p\}\times[-1,1]\subset F_0\times[-1,1]\subset M_0'\,\,\,\,\,\,\text{ and }\,\,\,\,\,\,\{p\}\times[-1,1]\subset F_0\times[-1,1]\subset R\times[-1,1].\]  Then \[\data_0 = (Y_0, R, r,m_0,\eta,\alpha)\] is a marked odd closure of $(M_0,\Gamma_0)$, where $r$ and $m_0$ are the obvious embeddings of $R\times[-1,1]$ and $M_0$ into $Y_0$. 

Note that $M'_0$ is also  a  preclosure of $(M,\Gamma)$ in a natural way, formed using  the auxiliary surface $F=F_0\cup R_+(S^1)$. It is then clear that \[\data = (Y_0, R, r,m,\eta,\alpha)\] is a marked odd closure of $(M,\Gamma)$, where $m$ is the restriction of $m_0$ to $M\subset M_0$. In particular, $\SHIt(-\data) = \SHIt(-\data_0)$. This leads to the following definition.

\begin{definition} 
\label{def:0handle}We define the $0$-handle attachment map \[\mathscr{H}_0:\SHItfun(-M,-\Gamma)\to \SHItfun(-M_0,-\Gamma_0)\] to be the morphism determined by the identity map  \[id_{-\data,-\data_0}:\SHIt(-\data)\to\SHIt(-\data_0).\] 

\end{definition}
To prove that $\mathscr{H}_0$ is independent of the choices made in its construction, we need  to show that if $\data_0,\data_0'$ are marked odd closures of $(M_0,\Gamma_0)$ constructed as above, and  $\data,\data'$ are the corresponding marked odd closures of $(M,\Gamma)$, then the  diagram   \[ \xymatrix@C=55pt@R=30pt{
\SHIt(-\data) \ar[r]^-{id_{-\data,-\data_0}} \ar[d]_{\Psit_{-\data,-\data'}} & \SHIt(-\data_0) \ar@<-1.5ex>[d]^{\Psit_{-\data_0,-\data_0'}} \\
\SHIt(-\data') \ar[r]_-{id_{-\data',-\data_0'}} & \SHIt(-\data_0')
} \] commutes, where $\Psit_{-\data,-\data'}$ and $\Psit_{-\data_0,-\data_0'}$ are the canonical isomorphisms relating the systems associated to different closures. But this follows from the fact that $\Psit_{-\data_0,-\data_0'}$ is a composition of  maps associated to 2-handle   and \emph{splicing} cobordisms, and $\Psit_{-\data,-\data'}$ can be defined via the exact same composition (refer to \cite{bs3} for the definition of these maps and \cite[Subsection 4.2]{bsSHM} for the same argument  in the  sutured monopole Floer context).

\subsection{1-handles}
\label{ssec:1handles}
Suppose $D_-$ and $D_+$ are disjoint embedded disks in $\partial M$ which each intersect $\Gamma$ in a single properly embedded arc. To attach a contact $1$-handle  to $(M,\Gamma)$ along these disks, we glue the   contact manifold $(D^2\times[-1,1],\xi_{D^2})$ to $(M,\Gamma)$ by diffeomorphisms \[D^2\times\{-1\}\to D_- \ \ {\rm and} \ \  D^2\times\{+1\}\to D_+,\] which preserve and reverse orientations, respectively, and  identify the dividing sets with the sutures, and then we round corners, as illustrated in Figure \ref{fig:onehandle}. Let $(M_1,\Gamma_1)$ be the resulting sutured  manifold. As in the $0$-handle case, it is not hard to construct a marked odd closure of $(M_1,\Gamma_1)$ which is also a marked odd closure  of $(M,\Gamma)$, so that we may define the contact $1$-handle attachment map to be the ``identity" map in this case as well. 

\begin{figure}[ht]

\centering
\includegraphics[width=14cm]{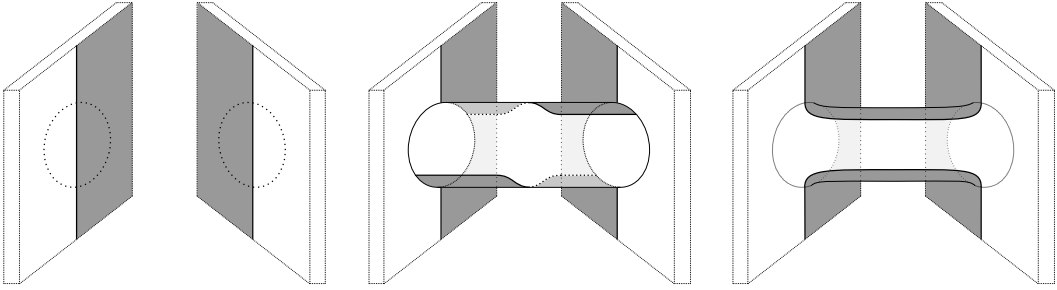}
\caption{Left, a collar neighborhood of a subsurface of  $\partial M$ containing the disks $D_-,D_+\subset \partial M$, whose boundaries are dotted. Middle, attaching the contact $1$-handle.  Right, the $1$-handle attachment after rounding corners.}
\label{fig:onehandle}
\end{figure}

Indeed, suppose $M_1'$ is a preclosure of $(M_1,\Gamma_1)$ formed from an auxiliary surface $F_1$. Then there are  natural identifications \[\partial_{\pm}M_1'=R_{\pm}(\Gamma_1)\cup F_0.\]  Let $R$ be a copy of $\partial_+M_1'$. Let $Y_1$ be the closed 3-manifold obtained by gluing $R\times[-1,1]$ to $M_1'$ by the ``identity" map from $R\times\{-1\}$ to $\partial_+M_1'$ and by a map from $R\times\{+1\}$ to $\partial_-M_1'$ which sends a point \[p\in F_1\subset R\,\,\,\,\,\,\text{ to }\,\,\,\,\,\,p\in F_1\subset  \partial_-M_1'.\] Let $\eta\subset R$ be an oriented, nonseparating curve contained in $F_1\subset R$ and let $\alpha\subset Y_1$ be the union of the oriented arcs \[\{p\}\times[-1,1]\subset F_1\times[-1,1]\subset M_1'\,\,\,\,\,\,\text{ and }\,\,\,\,\,\,\{p\}\times[-1,1]\subset F_1\times[-1,1]\subset R\times[-1,1].\]  Then \[\data_1 = (Y_1, R, r,m_1,\eta,\alpha)\] is a marked odd closure of $(M_1,\Gamma_1)$, where $r$ and $m_1$ are the obvious embeddings of $R\times[-1,1]$ and $M_1$ into $Y_1$.

In complete analogy with the $0$-handle case,  we note that $M'_1$ is also  a  preclosure of $(M,\Gamma)$, the point being that the union of $F_1\times[-1,1]$ with the contact $1$-handle is a product $F\times[-1,1]$, where $F$ is an auxiliary surface for $(M,\Gamma)$. It follows that \[\data = (Y_1, R, r,m,\eta,\alpha)\] is a marked odd closure of $(M,\Gamma)$, where $m$ is the restriction of $m_1$ to $M\subset M_1$. In particular, $\SHIt(-\data) = \SHIt(-\data_1)$. This leads to the following definition.

\begin{definition} 
\label{def:1handle}We define the $1$-handle attachment map \[\mathscr{H}_1:\SHItfun(-M,-\Gamma)\to \SHItfun(-M_1,-\Gamma_1)\] to be the morphism determined by the identity map  \[id_{-\data,-\data_1}:\SHIt(-\data)\to\SHIt(-\data_1).\] 
\end{definition}

The same reasoning as in the $0$-handle case shows that the map $\mathscr{H}_1$ is independent of the choices made in its construction.

\subsection{2-handles}
\label{ssec:2handles}
In this subsection, we define the map associated to contact 2-handle attachment. Along the way, we define a map associated to surgery on a framed knot in a sutured manifold. 

Suppose $\gamma$ is an embedded curve in $\partial M$ which intersects $\Gamma$ in two points. Let $A(\gamma)$ be an annular neighborhood of $\gamma$ intersecting $\Gamma$ in two cocores. To attach a contact $2$-handle  to $(M,\Gamma,\xi)$ along $\gamma$, we glue $(D^2\times[-1,1],\xi_{D^2})$ to $(M,\Gamma,\xi)$ by an orientation-reversing diffeomorphism \[\partial D^2\times [-1,1] \to A(\gamma)\] which identifies positive regions with negative regions, and then  round corners, as illustrated in Figure \ref{fig:twohandle2}. Let $(M_2,\Gamma_2)$ be the resulting sutured  manifold. We will show that there exists  a marked odd closure of $(M_2,\Gamma_2)$ which is obtained from  a marked odd closure of $(M,\Gamma)$ via integer surgery, and will accordingly define the $2$-handle attachment map to be the map induced by the 4-dimensional 2-handle cobordism corresponding to this surgery, roughly speaking.

\begin{figure}[ht]

\centering
\includegraphics[width=11.3cm]{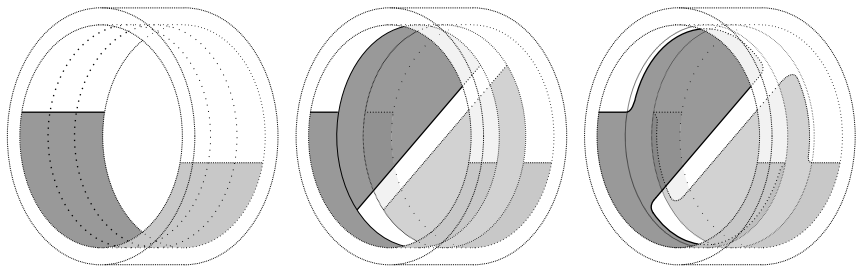}
\caption{Left, a collar neighborhood $N$ of a subsurface of $\partial M$ containing $A(\gamma)\subset\partial M$, whose boundary is dotted.  Middle,  attaching the contact $2$-handle. Right, the $2$-handle attachment after rounding corners.}
\label{fig:twohandle2}
\end{figure}

We construct the aforementioned closure of $(M_2,\Gamma_2)$ in a slightly roundabout way. Let us first consider the sutured manifold $(M_{1},\Gamma_{1})$  obtained from $(M_2,\Gamma_2)$ by attaching a contact $1$-handle along disks in the interiors of the $D^2\times\{\pm 1\}$ boundary components of the contact $2$-handle, as indicated in Figure \ref{fig:twohandle}.  Let \[\mathscr{H}_1:\SHMtfun(-M_2,-\Gamma_2)\to \SHMtfun(-M_1,-\Gamma_1)\] be the corresponding $1$-handle attachment map, as defined in Subsection \ref{ssec:1handles}. It is not hard to see that $(M_{1},\Gamma_{1})$ is diffeomorphic to the sutured manifold obtained from $(M,\Gamma)$ by performing $\partial M$-framed surgery on a parallel copy $\gamma'$ of $\gamma$ in the interior of $M$. 

\begin{figure}[ht]
\centering
\includegraphics[height=5.2cm]{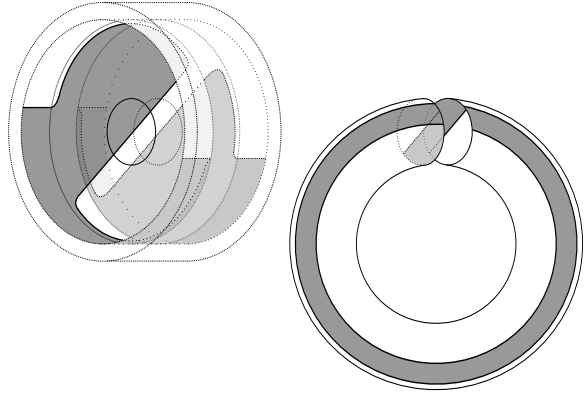}
\caption{Attaching a contact $1$-handle  to form $M_{1}$. The circles on $D^2\times\{\pm 1\}$ indicate   where the feet of this handle are to be attached. The union of the $1$-handle  below with the portion of $M_1$ shown above  is the solid torus $N_{1}$.}
\label{fig:twohandle}
\end{figure}

To be precise, let us suppose that $\gamma'$ is contained in the solid torus neighborhood $N\subset M$ shown in Figure \ref{fig:twohandle2}. Let $N_{1}\subset M_{1}$ be the solid torus obtained from $N$ by attaching the $1$- and $2$-handles  as indicated in Figures \ref{fig:twohandle2} and \ref{fig:twohandle}.  
Note that  \begin{equation}\label{eqn:NM'}(M\ssm N,\Gamma|_{M\ssm N})=(M_{1}\ssm N_{1},\Gamma_{1}|_{M_{1}\ssm N_{1}}).\end{equation} Furthermore, the restriction of the identity map on these complements  to $\partial M\ssm N  = \partial M_1\ssm N_1$ extends uniquely, up to isotopy, to a diffeomorphism   of pairs \[(\partial M, \Gamma)\to(\partial M_{1}, \Gamma_{1}).\] The identity map on the complement in \eqref{eqn:NM'} therefore extends naturally to a diffeomorphism \[(M\ssm N',\Gamma)\to(M_{1}\ssm N_{1}',\Gamma_{1}),\] where $N'\subset \inr(N)$ and $N_1'\subset \inr(N_1)$ are  slightly smaller solid tori. This provides a canonical, up to isotopy, diffeomorphism \begin{equation}\label{eq:f}f:(M',\Gamma')\to(M_{1},\Gamma_{1}),\end{equation} where $(M',\Gamma')$ is the sutured manifold obtained from $(M,\Gamma)$ via $\partial M$-framed surgery on $\gamma'$.

In order to define the contact $2$-handle map, we first define a morphism associated to this surgery. In fact, we take this opportunity to define a map associated to  surgery on any framed knot in the interior of a sutured manifold.

Suppose $\data = (Y,R,r,m,\eta,\alpha)$ is a marked odd closure of $(M,\Gamma)$. Suppose $K$ is a framed knot in the interior of $M$, and let $(M',\Gamma')$ be the sutured manifold obtained via surgery on $K$ with respect to this framing. Let $Y'$ be the  3-manifold obtained from $Y$ by performing surgery on  $m(K)$ with respect to the induced framing. Then  $\data' = (Y',R,r',m',\eta,\alpha)$ is a marked odd closure of $(M',\Gamma')$, where  $r'$ is the map induced by $r$ and $m'$ is the embedding of $M'$ into $Y'$ induced by $m$. Let $W$ be the $2$-handle cobordism from $Y$ to $Y'$  obtained from $Y\times[0,1]$ by attaching the 2-handle corresponding to the above surgery and let $\nu\subset W$ be the obvious cylindrical cobordism from \[(r(\eta\times\{0\})\sqcup \alpha)\subset Y\,\,\,\,\,\text{ to }\,\,\,\,\,(r'(\eta\times\{0\})\sqcup \alpha)\subset Y'.\] We define 
 \[F_{K}: \SHItfun(-M,-\Gamma)\to\SHItfun(-M',-\Gamma')\] 
  to be the morphism induced by the map 
\[I_*({-}W|{-}R)_{-\nu}:\SHIt(-\data)\to\SHIt(-\data').\] To prove that $F_{K}$ is well-defined, we must show that the  diagram 
\[ \xymatrix@C=41pt@R=32pt{
\SHIt(-\data_1) \ar[rr]^-{I_*({-}W_1|{-}R_1)_{-\nu_1}} \ar[d]_{\Psit_{-\data_{1},-\data_{2}}} && \SHIt(-\data_1') \ar[d]^{\Psit_{-\data_1',-\data_2'}} \\
\SHIt(-\data_{2}) \ar[rr]_-{I_*({-}W_2|{-}R_2)_{-\nu_2}} &&  \SHIt(-\data_2')
} \]
commutes, for any two marked odd closures $\data_1,\data_2$ of $(M,\Gamma)$, where  $\data_1',\data_2'$ are the induced marked odd closures of $(M',\Gamma')$. As explained in \cite[Subsection 4.2]{bsSHM} in the context of sutured monopole  homology, this diagram commutes because the cobordisms used to define these maps commute: $W_1$ and $W_2$ are built by attaching $2$-handles along curves in the regions $m_1(M)$ and $m_2(M)$, while the vertical isomorphisms are defined from cobordisms built by attaching 2-handles or splicing along tori outside of these regions. 

Let us now return to the situation at hand, where $(M_2,\Gamma_2)$ is obtained from $(M,\Gamma)$ by attaching a contact 2-handle along $\gamma$, and $f$ is the diffeomorphism in (\ref{eq:f}).

\begin{definition} 
\label{def:2handle}We define the $2$-handle attachment map \[\mathscr{H}_2:\SHItfun(-M,-\Gamma)\to \SHItfun(-M_2,-\Gamma_2)\] to be the composition $\mathscr{H}_2 = \mathscr{H}_1^{-1}\circ \SHItfun(f)\circ F_{\gamma'}.$
\end{definition}

That $\mathscr{H}_2$ is independent of $\gamma'$  follows from the fact that any two such parallel copies of $\gamma$ are related by an ambient isotopy of $M$ supported in $N$.  

\begin{remark}Unpacking the composition above, we see that $\mathscr{H}_2$ may  also be formulated  as follows. Suppose $\data=(Y,R,r,m,\eta,\alpha)$ is a marked odd closure of $(M,\Gamma)$ and let $\data'$ be the induced marked odd closure of the surgered manifold $(M',\Gamma')$ as above. Then \[\data_{2}=(Y',R,r',m_{2},\eta,\alpha)\] is a marked odd closure of $(M_{2},\Gamma_{2})$, where $m_{2}$ is the restriction of $m'\circ f^{-1}$ to $M_2\subset M_{1}$. Let \[id_{-\data',-\data_{2}}: \SHIt(-\data')\to\SHIt(-\data_{2})\] be the identity map on $\SHIt(-\data')=\SHIt(-\data_{2})$. Then $\mathscr{H}_2$ is the morphism induced by the map 
\[id_{-\data',-\data_{2}}\circ I_*({-}W|{-}R)_{-\nu}:\SHIt(-\data)\to \SHIt(-\data_2).\] In other words, the $2$-handle map is really just the map of systems induced by the cobordism map corresponding to  surgery along the curve of attachment.
\end{remark}

\subsection{3-handles}
\label{ssec:3handles}
Attaching a contact 3-handle to $(M,\Gamma)$ amounts to gluing the Darboux ball $(B^3,S^1,\xi_{std})$ to $(M,\Gamma)$ along an $S^2$ boundary component of $M$ with one suture, identifying positive regions with negative regions, and vice versa. Let $(M_3,\Gamma_3)$ be the result of this gluing. We will  assume that $\partial M$ is disconnected, so that $M_3$ has boundary. Let $p$ be a point in $M_3$ in the interior of the Darboux ball we glued in. Then there is a canonical isotopy class of diffeomorphisms \[f:(M,\Gamma)\to (M',\Gamma'),\] where $(M',\Gamma')$ is the sutured  manifold obtained by taking the  connected sum of $(M_3,\Gamma_3)$ with $(B^3,S^1)$ at the point $p$. Let $(M_{0},\Gamma_{0})$ be the disjoint union of $(M_3,\Gamma_3)$ with $(B^3,S^1),$ and let \[\mathscr{H}_0: \SHItfun(-M_3, -\Gamma_3)\to \SHItfun(-M_{0},-\Gamma_{0})\] be the corresponding  $0$-handle attachment map, as defined in Subsection \ref{ssec:0handles}. Suppose \[\data_{0} = (Y_{0},R,r,m,\eta,\alpha)\] is a marked odd closure of $(M_{0},\Gamma_{0})$. Then \[\data' = (Y',R,r,m',\eta,\alpha)\] is a marked odd closure of $(M',\Gamma')$, where $Y'$ is the self  connected sum obtained from $Y_{0}$ by removing Darboux balls around $m(p)$ and some point in $m(B^3)\subset Y_{0}$ and gluing in $S^2\times I$, and $m'$ is the  embedding of $M'$ into $Y'$ induced by $m$. In particular, $Y'$ is a  connected sum of $Y_{0}$ with  $S^1\times S^2$.
Let $W$ be the  natural  $1$-handle cobordism from $Y_0$ to $Y'$, and let $\nu\subset W$ be  the natural  cylindrical cobordism from \[(r(\eta\times\{0\})\sqcup \alpha)\subset Y_0\,\,\,\,\,\text{ to }\,\,\,\,\,(r(\eta\times\{0\})\sqcup \alpha)\subset Y'.\] Let \[F_{\#}:\SHItfun(-M',-\Gamma')\to \SHItfun(-M_0,-\Gamma_0)\] be the morphism determined by the map \[I_*(W|{-}R)_{\nu}: \SHIt(-\data')\to\SHIt(-\data_{0}).\]

\begin{definition} 
\label{def:3handle}We define the $3$-handle attachment map \[\mathscr{H}_3:\SHItfun(-M,-\Gamma)\to \SHItfun(-M_3,-\Gamma_3)\] to be  the composition  $\mathscr{H}_3 = \mathscr{H}_0^{-1}\circ F_{\#}\circ\SHItfun(f).$
\end{definition}

To show that this map is well-defined, we only need to argue that $F_{\#}$ is well-defined. But this follows from same sort of reasoning as was used to argue that $F_K$ is well-defined: namely,  the  1-handle cobordism used to define $F_{\#}$ is formed via $1$-handle attachment along balls in the interiors of $Y'$ and $Y_0$ and therefore commutes with the 2-handle and splicing cobordisms used to define the canonical isomorphisms in the systems $\SHItfun(-M',-\Gamma')$ and $\SHItfun(-M_0,-\Gamma_0)$. 




\subsection{A further property}
Below, we prove a  lemma which will be useful for defining the contact invariant in Section \ref{sec:instantoninvt}. Suppose \[f:(M,\Gamma)\to (M',\Gamma')\] is a  diffeomorphism and $(M_i,\Gamma_i)$ is obtained from $(M,\Gamma)$ by attaching a contact $i$-handle along an attaching region $S\subset \partial M$. Note that $f$ extends uniquely, up to isotopy, to a sutured diffeomorphism \[\bar f:(M_i,\Gamma_i)\to(M'_i,\Gamma_i'),\] where $(M_i',\Gamma_i')$ is obtained from $(M',\Gamma')$ by attaching a contact $i$-handle along the attaching region $f(S)\subset \partial M'$.  
Then we have the following.

\begin{lemma}
\label{lem:nathandle}
The diagram 
\[ \xymatrix@C=30pt@R=35pt{
\SHItfun(-M,-\Gamma)   \ar[r]^-{ \mathscr{H}_i} \ar[d]_{\SHItfun(f)} &\SHItfun(-M_i,-\Gamma_i) \ar[d]^{\SHItfun(\bar f)}\\
\SHItfun(-M',-\Gamma') \ar[r]_-{\mathscr{H}'_i}  & \SHItfun(-M'_i,-\Gamma'_i) }\] 
 commutes, where $\mathscr{H}_i$ and $\mathscr{H}_i'$ are the appropriate contact $i$-handle attachment maps.
\end{lemma}

\begin{proof}
The composition $\SHItfun(\bar f)\circ \mathscr{H}_i$ is ultimately defined in terms of the map on instanton Floer homology induced by  a natural cobordism (the identity cobordism, a 1-handle cobordism, or a 2-handle cobordism) from a closure of $(-M,-\Gamma)$ to a closure of $(-M_i',-\Gamma_i')$. Unraveling definitions, it is clear that the composition $\mathscr{H}_i'\circ\SHItfun(f)$ is determined by the same cobordism map.
\end{proof}

\section{A contact invariant in sutured instanton homology}
\label{sec:instantoninvt}
In this section, we use the relative Giroux correspondence to define the contact invariant 
\[\theta(M,\Gamma,\xi)\in\SHItfun(-M,-\Gamma)\] outlined in the introduction. We then establish some basic properties of this invariant, such as the fact that it vanishes for overtwisted contact structures and is nonzero for the complement of a Darboux ball in a Stein fillable contact manifold.

\subsection{The contact invariant}
Suppose $(M,\Gamma,\xi)$ is a sutured contact manifold with partial open book decomposition $(S,P,h,\mathbf{c},f)$. Recall that $M(S,P,h,\mathbf{c})$ is obtained from $H(S)$ by attaching contact 2-handles along the curves in the set $\boldsymbol{\gamma}(h,\mathbf{c})$ defined in \eqref{eqn:gammac}. Let \[\mathscr{H}:\SHItfun(-H(S))\to\SHItfun(-M(S,P,h,\mathbf{c}))\] be corresponding composition of contact $2$-handle attachment morphisms. 

\begin{definition}
\label{def:contactinvtobinstanton}  We define \[\theta(S,P,h,\mathbf{c},f):=\SHItfun(f)(\mathscr{H}(\mathbf{1}))\in\SHItfun(-M,-\Gamma),\] where $\mathbf{1}$ is the generator of $\SHItfun(-H(S))\cong \C$.
\end{definition}

\begin{definition}
We define  \[\theta(M,\Gamma,\xi):=\theta(S,P,h,\mathbf{c},f)\in \SHItfun(-M,-\Gamma)\] for any partial open book decomposition $(S,P,h,\mathbf{c},f)$ of $(M,\Gamma,\xi)$.
\end{definition}

That the element $\theta(M,\Gamma,\xi)$ is well-defined is the content of  the following theorem.

\begin{theorem}
\label{thm:well-defined4} The element $\theta(S,P,h,\mathbf{c},f)$ is independent of the partial open book decomposition $(S,P,h,\mathbf{c},f)$ of $(M,\Gamma,\xi)$. 
\end{theorem}

The rest of this subsection is devoted to the proof of Theorem \ref{thm:well-defined4}. As a first step,   we have the following lemma.

\begin{lemma}
\label{lem:isomorphic}
If $(S,P,h,\mathbf{c},f)$ and $(S',P',h',\mathbf{c}',f')$ are isomorphic partial open book decompositions, then $\theta(S,P,h,\mathbf{c},f) = \theta(S',P',h',\mathbf{c}',f')$. 
\end{lemma}

\begin{proof}
We must show that \begin{equation}\label{eqn:isoeq}\SHItfun(f)(\mathscr{H}(\mathbf 1)) = \SHItfun(f')(\mathscr{H}'(\mathbf 1')),\end{equation} where $\mathscr{H}$ and $\mathscr{H'}$ are the compositions of contact 2-handle maps used to define $\theta(S,P,h,\mathbf{c},f)$ and $\theta(S',P',h',\mathbf{c}',f')$ and $\mathbf 1$ and $\mathbf 1'$ are the generators of $\SHItfun(-H(S))$ and $\SHItfun(-H(S'))$.

 Since these open book decompositions are isomorphic, there exist maps  $\tilde g$ and $\bar{\tilde g}$ as in (\ref{eqn:tildeg}) and (\ref{eqn:barg}) such that $f=f'\circ \bar{\tilde g}$. Note that we have a commutative diagram 
 \[ \xymatrix@C=40pt@R=35pt{
\SHItfun(-H(S))   \ar[r]^-{ \mathscr{H}} \ar[d]_{\SHItfun(\tilde g)} &\SHItfun(-M(S,P,h,\mathbf{c})) \ar[d]^{\SHItfun(\bar {\tilde g})} \ar[r]^-{ \SHItfun(f)}&\SHItfun(-M,-\Gamma)  \ar[d]_{id}\\
\SHItfun(-H(S')) \ar[r]_-{\mathscr{H}'}  & \SHItfun(-M(S',P',h',\mathbf{c}'))\ar[r]^-{ \SHItfun(f')}&\SHItfun(-M,-\Gamma)  }.\] The leftmost square commutes by Lemma \ref{lem:nathandle} and the rightmost square commutes since \[\SHItfun(f) = \SHItfun(f'\circ \bar{\tilde g}) = \SHItfun(f')\circ\SHItfun(\bar{\tilde g}).\] The equality in (\ref{eqn:isoeq}) then follows as long as $\SHItfun(\tilde g)$ sends $\mathbf 1$ to $\mathbf 1'$, but it does since this map is an isomorphism and $\mathbf 1$ and $\mathbf 1'$ are the generators.
\end{proof}

Since isomorphic partial open book decompositions give rise to the same contact element, by Lemma \ref{lem:isomorphic}, it suffices, for the proof of Theorem \ref{thm:well-defined4}, to establish the following.

\begin{proposition}
\label{prop:invariancestab}
If the partial open book decomposition $(S',P',h',\mathbf{c}',f')$ is a positive stabilization of $(S,P,h,\mathbf{c},f)$, then $\theta(S',P',h',\mathbf{c}',f') = \theta(S,P,h,\mathbf{c},f).$
\end{proposition}

\begin{proof}
Suppose the partial open book decomposition $(S',P',h',\mathbf{c}',f')$ of $(M,\Gamma,\xi)$ is a positive stabilization of $(S,P,h,\mathbf{c},f)$. Let 
\begin{align*}
\mathscr{H}&:\SHItfun(-H(S))\to \SHItfun(-M(S,P,h,\mathbf{c}))\\
\mathscr{H}'&:\SHItfun(-H(S'))\to \SHItfun(-M(S',P',h',\mathbf{c}'))
\end{align*} be the compositions of contact 2-handle  maps used to define the elements $\theta(S,P,h,\mathbf{c},f)$ and $\theta(S',P',h',\mathbf{c}',f')$. To prove Proposition \ref{prop:invariancestab}, we must show that \begin{equation}\label{eqn:compeq}\SHItfun(f)(\mathscr{H}(\mathbf{1}))=\SHItfun(f')(\mathscr{H}'(\mathbf{1}')),\end{equation} where $\mathbf{1}$ and $\mathbf{1}'$ are  the generators of $\SHItfun(-H(S))$ and $\SHItfun(-H(S')).$ Let \[\mathscr{H}^{c_0}:\SHItfun(-H(S'))\to \SHItfun(-M(S',P',h',c_0))\] be the morphism associated to the $2$-handle attachment along $\gamma_0$ and let 
\[\mathscr{H}^{c_{>0}}:\SHItfun(-M(S',P',h',c_0))\to \SHItfun(-M(S',P',h',\mathbf c'))\] be the morphism associated to the composition of $2$-handle attachments along $\gamma_1',\dots,\gamma_n',$ so that \[\mathscr{H'} = \mathscr{H}^{c_{>0}}\circ  \mathscr{H}^{c_0}.\] Finally, let $q$ and $\bar q$ be the contactomorphisms in \eqref{eqn:q} and \eqref{eqn:tildeq}, so that $f'=f\circ \bar q$.
Then we have 
\[\SHItfun(f')\circ\mathscr{H}' = \SHItfun(f)\circ\SHItfun(\bar q)\circ \mathscr{H}^{c_{>0}}\circ  \mathscr{H}
=\SHItfun(f)\circ \mathscr{H}\circ\SHItfun(q)\circ  \mathscr{H}^{c_0},\] where the second equality is an application of  Lemma \ref{lem:nathandle}.  Thus, for \eqref{eqn:compeq}, it suffices to show that \[\SHItfun(q)(\mathscr{H}^{c_0}(\mathbf{1}'))=\mathbf{1},\] which is equivalent to proving that $\mathscr{H}^{c_0}$ is nonzero. 

By definition, the curve $\gamma_0\subset \partial H(S')$ is obtained from the curve \[(c_0\times\{1\})\cup (\partial c_0\times [-1,1])\cup (D_{\beta}(c_0)\times\{-1\})\,\subset\, \partial (S'\times[-1,1]),\]  shown in Figure \ref{fig:stab}, by rounding corners. Suppose  $\data = (Y,R,r,m,\eta,\alpha)$ is a marked odd closure of $H(S')$, let $\gamma_0'$ be a parallel copy of $\gamma_0$ in the interior of $Y$, and let  $Y'$ be the result of $0$-surgery on $m(\gamma_0')$ with respect to the framing  induced by $\partial H(S')$. By the construction of the contact 2-handle  map in the previous section, we know that there is an embedding \[m':M(S',P',h',c_0)\to Y'\] such that $\data'=(Y',R,r,m',\eta,\alpha)$ is a marked odd closure of $M(S',P',h',c_0)$. Let $W$ be the 2-handle cobordism from $Y$ to $Y'$ obtained from $Y\times[0,1]$ by attaching a $2$-handle corresponding to this surgery. Then $\mathscr{H}^{c_0}$ is the morphism determined by  the induced map
\begin{equation}\label{eqn:Wmapiso}I_*({-}W|{-}R)_{-\nu}:\SHIt(-\data)\to \SHIt(-\data'),\end{equation}
where $\nu\subset W$ is the obvious cylindrical cobordism from \[(r(\eta\times\{0\})\sqcup \alpha)\subset Y\,\,\,\,\,\text{ to }\,\,\,\,\,(r'(\eta\times\{0\})\sqcup \alpha)\subset Y'.\]
\begin{figure}[ht]
\labellist
\small
\pinlabel $\beta$ at 100 73
\pinlabel $c_0$ at 60 -6
\endlabellist
\centering
\includegraphics[width=12.5cm]{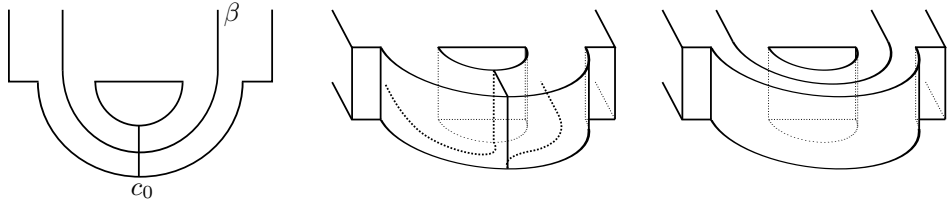}
\caption{Left, the surface $S'$ with the cocore $c_0$ of the 1-handle $H_0$ and the curve $\beta$. Middle, the curve $\gamma_0$ in $H(S')$ drawn as $(c_0\times\{1\})\cup (\partial c_0\times [-1,1])\cup (D_{\beta}(c_0)\times\{-1\})$ in $S'\times[-1,1]$. Right, the curve  $\beta'$ drawn as $\beta\times\{1\}\subset S'\times\{1\}$.}
\label{fig:stab}
\end{figure}

Note that $\gamma_0$ is isotopic to the curve $\beta'\subset\partial H(S')$ corresponding to  $\beta\times\{1\}\subset S'\times\{1\}$, by an isotopy which sends the $\partial H(S')$-framing on $\gamma_0$  to the $(\partial H(S')+1)$-framing on $\beta'$. Since $\beta'$ is contained in the positive region of $\partial H(S')$, the image  $m(\beta')$ is isotopic to $r(b\times\{t\})$ for some embedded curve $b\subset R$ and any $t\in [0,1]$, by an isotopy which sends the $\partial H(S')$-framing on $m(\beta')$ to the $r(R\times\{t\})$-framing on $r(b\times\{t\})$. We may therefore think of $W$ as the cobordism associated to $(+1)$-surgery on $r(b\times\{t\})$. But this is exactly the sort of cobordism used to define the canonical isomorphisms relating the sutured instanton homologies associated to different closures of a sutured manifold, as described in \cite[Section 9]{bs3}. In particular, the map in \eqref{eqn:Wmapiso} is an isomorphism, proving that $\mathscr{H}^{c_0}$ is nonzero.
\end{proof}

This completes the proof of Theorem \ref{thm:well-defined4}. \qed

\subsection{Properties} Below, we establish some properties of the invariant $\theta(M,\Gamma,\xi)$. The first result below says that the invariant $\theta$ behaves functorially with respect to contact $(+1)$-surgery.

\begin{proposition}
\label{prop:contactsurgeryrelativegirouxinstanton}
Suppose $K$ is a Legendrian knot in the interior of $(M,\Gamma,\xi)$ and  that $(M',\Gamma',\xi')$ is the result of contact $(+1)$-surgery on $K$. Then the map \[F_K:\SHItfun(-M,-\Gamma)\to\SHItfun(-M',-\Gamma')\] corresponding to this surgery, as defined in Subsection \ref{ssec:2handles}, sends $\theta(M,\Gamma,\xi)$ to $\theta(M',\Gamma',\xi')$.
\end{proposition}

\begin{proof}
Let $(S,P,h,\mathbf{c},f)$ be a partial open book decomposition of $(M,\Gamma,\xi)$ such that $K =f(\Lambda)$, where $\Lambda$ is a pushoff of \[\lambda\times\{-1\}\subset S\times[-1,1]\subset M(S,P,h,\mathbf{c})\] into the interior of $M(S,P,h,\mathbf{c})$, where $\lambda\subset P$ is a curve which intersects $c_1\in\mathbf{c}$ in a single point, is disjoint from all other $c_i$, and is not homotopic to $\partial P$. We further require that the contact framing on $K$ agrees with the contact framing on $\Lambda$ (which is induced by the $S$-framing on $\lambda$).
One can construct an $(S,P,h,\mathbf{c},f)$ with these properties by including $K$ in the Legendrian graph used to define the partial open book, as described in \cite{hkm4}. 

Let $H'(S)$ and $M'(S,P,h,\mathbf{c})$  be the contact manifolds obtained from $H(S)$ and $M(S,P,h,\mathbf{c})$, respectively, by performing contact $(+1)$-surgery on $\Lambda$. The contactomorphism $f$ naturally induces a  contactomorphism \[\bar f:M'(S,P,h,\mathbf{c})\to (M',\Gamma',\xi')\] such that the diagram  \
\[ \xymatrix@C=50pt@R=35pt{
\SHItfun(-M(S,P,h,\mathbf{c}))   \ar[r]^-{ \SHItfun(f)} \ar[d]_{F_\Lambda} &\SHMtfun(-M,-\Gamma) \ar[d]^{F_K}\\
\SHItfun(-M'(S,P,h,\mathbf{c})) \ar[r]_-{\SHItfun(\bar f)}  & \SHMtfun(-M',-\Gamma') }\] 
 commutes, by the same sort of argument as was used in the proof of Lemma \ref{lem:nathandle}. Note that there is a canonical isotopy class of contactomorphism \[g:H'(S)\to H(S)\] which sends the attaching set \[\boldsymbol{\gamma}(h,\mathbf{c})\subset \partial H'(S)\,\,\,\,\,\,{\rm to}\,\,\,\,\,\,\boldsymbol{\gamma}(h\circ D_{\lambda}^{-1},\mathbf{c})\subset \partial H(S).\] This map naturally induces a contactomorphism  \[\bar g:M'(S,P,h,\mathbf{c})\to M(S,P,h\circ D_{\lambda}^{-1},\mathbf{c}).\] Thus, \[(S,P,h\circ D_{\lambda}^{-1},\mathbf{c},f_\lambda:=\bar f\circ (\bar g)^{-1})\] is a partial open book decomposition for $(M',\Gamma',\xi')$.
 
 Let 
\begin{align*}
\mathscr{H}&:\SHItfun(-H(S))\to \SHItfun(-M(S,P,h,\mathbf{c}))\\
\mathscr{H}'&:\SHItfun(-H(S'))\to \SHItfun(-M'(S,P,h,\mathbf{c}))\\
\mathscr{H}_{\lambda}&:\SHItfun(-H(S))\to \SHItfun(-M(S,P,h\circ D_{\lambda}^{-1},\mathbf{c}))
\end{align*} be the compositions of contact 2-handle maps associated to the attaching sets 
\[\boldsymbol{\gamma}(h,\mathbf{c})\subset \partial H(S)\,\,\,\,\, \,{\rm and}\,\,\,\,\,\,
\boldsymbol{\gamma}(h,\mathbf{c})\subset \partial H'(S)\,\,\,\,\, \,{\rm and}\,\,\,\,\,\,
\boldsymbol{\gamma}(h\circ D_{\lambda}^{-1},\mathbf{c})\subset \partial H(S),\]
respectively. The commutativity of the diagram
  \[ \xymatrix@C=50pt@R=35pt{
\SHItfun(-M(S,P,h,\mathbf{c}))   \ar[r]^-{\SHItfun(f)} \ar[d]_{\SHItfun(\bar g)\,\circ\, F_\Lambda } & \SHItfun(-M,-\Gamma) \ar[d]^{F_K}\\
  \SHItfun(-M(S,P,h\circ D_{\lambda}^{-1},\mathbf{c})) \ar[r]_-{\SHItfun(f_\lambda)}  & \SHItfun(-M',-\Gamma')}\] follows immediately from that of the previous diagram combined with the fact that \[ \SHItfun(\bar f) = \SHItfun(f_\lambda)\circ\SHItfun(\bar g).\]
Since 
\begin{align*}
\SHItfun(f)(\mathscr{H}(\mathbf{1})) &= \theta(M,\Gamma,\xi),\\
\SHItfun(f_\lambda)(\mathscr{H}_\lambda(\mathbf{1})) &= \theta(M',\Gamma',\xi'),
\end{align*}
by definition, 
it suffices for the proof of the proposition to show that \begin{equation}\label{eqn:keyeq}(\SHItfun(\bar g)\circ F_{\Lambda})(\mathscr{H}(\mathbf{1}))= \mathscr{H}_{\lambda}(\mathbf{1}).\end{equation} For this, consider the diagram 
 \[ \xymatrix@C=30pt@R=35pt{
\SHItfun(-H(S))   \ar[r]^-{\mathscr{H}} \ar[d]_{F_\Lambda } & \SHItfun(-M(S,P,h,\mathbf{c})) \ar[d]^{F_\Lambda}\\
\SHItfun(-H'(S))   \ar[r]^-{\mathscr{H}'} \ar[d]_{\SHItfun(g) } & \SHItfun(-M'(S,P,h,\mathbf{c})) \ar[d]^{\SHItfun(\bar g)}\\
  \SHItfun(-H(S)) \ar[r]_-{\mathscr{H}_\lambda}  & \SHItfun(-M(S,P,h\circ D_{\lambda}^{-1},\mathbf{c})).
  }\] 
The top square commutes since the 2-handle cobordisms between closures used to define the maps obviously commute. Moreover, the leftmost map \[F_{\Lambda}:\SHItfun(-H(S))\to \SHItfun(-H'(S))\] is induced by the same sort of 2-handle cobordism that defines the canonical isomorphisms between different closures of the same genus, since $\Lambda$ is isotopic to a curve contained in the negative region of $\partial H(S)$ (see \cite[Section 9]{bs3}). In particular, it is an isomorphism, and therefore sends $\mathbf{1}$ to $\mathbf{1}.$ It follows that the rightmost map \[F_{\Lambda}:\SHItfun(-M(S,P,h,\mathbf{c}))\to \SHItfun(-M'(S,P,h,\mathbf{c}))\] satisfies \begin{equation}\label{eqn:Lambda}F_{\Lambda}(\mathscr{H}(\mathbf{1})) = \mathscr{H}'(\mathbf{1}).\end{equation} The bottom square in the diagram commutes by Lemma \ref{lem:nathandle}, and since $\SHItfun(g)$ is an isomorphism, it sends $\mathbf{1}$ to $\mathbf{1}.$ Hence, \begin{equation}\label{eqn:Lambda2}\SHItfun(\bar g)(\mathscr{H}'(\mathbf{1})) = \mathscr{H}_\lambda(\mathbf{1}).\end{equation} Putting \eqref{eqn:Lambda} and \eqref{eqn:Lambda2} together, we obtain \eqref{eqn:keyeq}, completing the proof of Proposition \ref{prop:contactsurgeryrelativegirouxinstanton}.
\end{proof}

Next, we show that the invariant $\theta$ behaves as one would expect with respect to contactomorphism.

\begin{proposition}
\label{prop:contactomorphism} Suppose \[g:(M,\Gamma,\xi)\to(M',\Gamma',\xi')\] is a contactomorphism. Then the map \[\SHItfun(g): \SHItfun(-M,-\Gamma)\to\SHItfun(-M',-\Gamma')\] sends $\theta(M,\Gamma,\xi)$ to $\theta(M',\Gamma',\xi')$.
\end{proposition}

\begin{proof} 
Suppose $(S,P,h,\mathbf{c},f)$ is a partial open book decomposition for $(M,\Gamma,\xi)$. Then clearly $(S,P,h,\mathbf{c},g\circ f)$ is a partial open book decomposition for $(M',\Gamma',\xi')$. Letting \[\mathscr{H}:\SHItfun(-H(S))\to\SHItfun(-M(S,P,h,\mathbf{c}))\] be the corresponding composition of contact 2-handle maps, we have that \[\theta(M',\Gamma',\xi') := \SHItfun(g\circ f)(\mathscr{H}(\mathbf{1})) = \SHItfun(g)(\SHItfun(f)(\mathscr{H}(\mathbf{1}))) = \SHItfun(g)(\theta(M,\Gamma,\xi)),\] as desired.
\end{proof}

As explained in the introduction, the contact invariant $\theta$ behaves naturally with respect to the maps induced by handle attachments.

\begin{theorem}
\label{thm:handletheta2} Suppose  $(M_i,\Gamma_i,\xi_i)$  is obtained from $(M,\Gamma,\xi)$ by attaching a contact $i$-handle and $\mathscr{H}_i$ is the associated contact handle attachment map for $i=0,$ $1,$ or $2$. Then  \[\mathscr{H}_i:\SHItfun(-M,-\Gamma)\to\SHItfun(-M_i,-\Gamma_i)\] sends $\theta(M,\Gamma,\xi)$ to $\theta(M_i,\Gamma_i,\xi_i).$
\end{theorem}

\begin{proof} 
Let us first assume that $i=0$ and let us adopt all the notation from Subsection \ref{ssec:0handles}. Suppose $(S,P,h,\mathbf{c},f)$ is a partial open book decomposition for $(M,\Gamma,\xi)$. Then $(S', P,h,\mathbf{c}, f')$ is a partial open book decomposition of $(M_0,\Gamma_0,\xi_0)$, where $S'$ is the disjoint union of $S$ with $D^2$ and $f'$ is the disjoint union of $f$ with a contactomorphism  \[H(D^2)\to (B^3,S^1,\xi_{std}).\] Consider the diagram  
 \[ \xymatrix@C=30pt@R=35pt{
\SHItfun(-H(S))   \ar[r]^-{\mathscr{H}_0''} \ar[d]_{\mathscr{H} } & \SHItfun(-H(S'))  \ar[d]^{\mathscr{H}'}\\
\SHItfun(-M(S,P,h,\mathbf{c},f)) \ar[r]^-{\mathscr{H}_0'} \ar[d]_{\SHItfun(f) } & \SHItfun(-M(S',P,h,\mathbf{c},f')) \ar[d]^{\SHItfun(f')}\\
  \SHItfun(-M,-\Gamma) \ar[r]_-{\mathscr{H}_0}  & \SHItfun(-M_0,-\Gamma_0).
  }\] where $\mathscr{H}, \mathscr{H}'$ are the compositions of contact $2$-handle maps of the sort used to define $\theta$, and $\mathscr{H}_0, \mathscr{H}_0',\mathscr{H}_0''$ are the obvious contact $0$-handle maps. Since the map $\mathscr{H}_0''$ is an isomorphism (and therefore sends $\mathbf{1}$ to $\mathbf{1}$), we need only check that this diagram commutes. But this is  straightforward from the definitions of these maps---on the level of closures, the identity and 2-handle cobordisms defining these maps commute. The map $\mathscr{H}_0$ thus preserves the contact invariant as desired.

Let us now assume that $i=1$ and adopt all the notation from Subsection \ref{ssec:1handles}. The proof in this case is similar. We can find a partial open book decomposition $(S,P,h,\mathbf{c},f)$ for $(M,\Gamma,\xi)$ such that $(S',P,h,\mathbf{c},f')$ is a partial open book decomposition for $(M_1,\Gamma_1,\xi_1)$, where $S'$ is the surface obtained  by attaching a $1$-handle to $S$ away from $P$, and \[f':M(S',P,h,\mathbf{c})\to (M_0,\Gamma_0,\xi_0)\] is a contactomorphism which restricts to $f$ on $M(S,P,h,\mathbf{c})\subset M(S',P,h,\mathbf{c})$. (To find open book decompositions with this property, we first construct a partial open book decomposition for $(M_0,\Gamma_0,\xi_0)$ from a contact cell decomposition whose Legendrian graph  contains the core of the contact $1$-handle. We can then arrange that the resulting partial open book decomposition is precisely of the form $(S',P,h,\mathbf{c},f')$, where $(S,P,h,\mathbf{c},f)$ is a partial open book decomposition for $(M,\Gamma,\xi)$, as described above.) As in the previous case, it  suffices to check that the diagram  
 \[ \xymatrix@C=30pt@R=35pt{
\SHItfun(-H(S))   \ar[r]^-{\mathscr{H}_1''} \ar[d]_{\mathscr{H} } & \SHItfun(-H(S'))  \ar[d]^{\mathscr{H}'}\\
\SHItfun(-M(S,P,h,\mathbf{c},f)) \ar[r]^-{\mathscr{H}_1'} \ar[d]_{\SHItfun(f) } & \SHItfun(-M(S',P,h,\mathbf{c},f')) \ar[d]^{\SHItfun(f')}\\
  \SHItfun(-M,-\Gamma) \ar[r]_-{\mathscr{H}_1}  & \SHItfun(-M_1,-\Gamma_1).
  }\] commutes, where $\mathscr{H}, \mathscr{H}'$ are the compositions of contact $2$-handle maps of the sort used to define $\theta$, and $\mathscr{H}_1, \mathscr{H}_1',\mathscr{H}_1''$ are the obvious contact $1$-handle maps. Again, this commutativity is  straightforward from the definitions of these maps. The map $\mathscr{H}_1$ thus preserves the contact invariant as desired.

  Let us now assume that $i=2$ and adopt all the notation from Subsection \ref{ssec:2handles}. The contact $2$-handle attachment map \[\mathscr{H}_2:\SHItfun(-M,-\Gamma)\to\SHItfun(-M_2,-\Gamma_2)\] is defined by $\mathscr{H}_2=\mathscr{H}_1^{-1}\circ\SHItfun(f)\circ F_{\gamma'}$. We have shown that  $\mathscr{H}_1$ preserves the contact invariant; we can assume that $\gamma'$ is Legendrian so that $(M',\Gamma',\xi')$ is obtained from $(M,\Gamma,\xi)$ by contact $(+1)$-surgery on $\gamma'$, which means that $F_{\gamma'}$ preserves the contact invariant, by Proposition \ref{prop:contactsurgeryrelativegirouxinstanton}; and, finally, we can assume that $f$ is a contactomorphism \[f:(M',\Gamma',\xi')\to(M_1,\Gamma_1,\xi_1)\]  (see the discussion in \cite[Subsubsection 4.2.3]{bsSHM}) and therefore preserves the contact invariant by Proposition \ref{prop:contactomorphism}. The map $\mathscr{H}_2$ thus preserves the contact invariant as desired.
\end{proof}

\begin{remark} 
\label{rmk:embedding}Suppose $(M,\Gamma)$ is  a \emph{sutured submanifold} of $(M',\Gamma')$, as defined in \cite{hkm5}. Let $\xi$ be a contact structure on $M'\ssm \inr(M)$ with convex boundary and dividing set $\Gamma$ on $\partial M$ and $\Gamma'$ on $\partial M'$. As explained in Subsection \ref{ssec:future}, the sutured contact manifold $(M'\ssm \inr(M),\Gamma \cup \Gamma',\xi')$ can be obtained from a vertically invariant contact structure on $\partial M\times I$ by attaching contact handles. Given a contact handle decomposition $H$ of this sort, we  define  \[\Phi_{\xi,H}:\SHItfun(-M,-\Gamma)\to\SHItfun(-M',-\Gamma')\] to be the  corresponding composition of  contact handle attachment maps, as in the introduction. Note that if the contact handles in $H$ are $0$-, $1$-, and $2$-handles only and  if $\xi_M$ is a contact structure on $M$ which agrees with $\xi$ near $\partial M$, then \[\Phi_{\xi,H}(\theta(M,\Gamma,\xi_M)) = \theta(M',\Gamma',\xi_M\cup \xi)\] by Theorem  \ref{thm:handletheta2}.  
\end{remark}

Next, we show that $\theta$ vanishes for overtwisted contact structures.

\begin{theorem}
\label{thm:otinstanton}
If $(M,\Gamma,\xi)$ is overtwisted, then $\theta(M,\Gamma,\xi)=0$.
\end{theorem}

\begin{proof}

Let $N\subset M$ be a neighborhood of an overtwisted disk $D$. Take a Darboux ball in $N\ssm D$ and let $K$ be a Legendrian right-handed trefoil  in this ball with $tb(K)=1$ and $rot(K)=0$. Then the connected sum $K' = K\# \partial D$ is a Legendrian trefoil with $tb(K')=2$, and it has a connected Seifert surface $\Sigma\subset N\subset M$ of genus $1$.

Let $(M_-,\Gamma_-,\xi_-)$ be the result of contact $(-1)$-surgery on $K'$. Suppose $\data = (Y,R,r,m,\eta,\alpha)$ is a marked odd closure of $(M,\Gamma)$ and let $\data_-= (Y_-,R,r,m_-,\eta,\alpha)$ be the induced closure of $(M_-,\Gamma_-,\xi_-)$, where $Y_-$ is obtained from $Y$ via  contact $(-1)$-surgery on $m(K')$. Let $X$ be the associated $2$-handle cobordism from $Y$ to $Y_-$. Now, $(M,\Gamma,\xi)$ can be thought of as being obtained from $(M_-,\Gamma_-,\xi_-)$ via contact $(+1)$-surgery on a Legendrian pushoff $K''\subset M_-$ of $K'$. The associated $2$-handle cobordism from $Y_-$ to $Y$ is isomorphic to $-X$. The morphism \[F_{K''}:\SHItfun(-M_-,-\Gamma_-)\to\SHItfun(-M,-\Gamma)\] is therefore the equivalence class of the map associated to  $X$, viewed as a cobordism from $-Y_-$ to $-Y$. 

We can cap off $\Sigma$ to a closed surface $\Sigma' \subset X$ of genus $1$ with self-intersection \[\Sigma'\cdot\Sigma' = tb(K')-1=1.\] This surface violates the adjunction inequality $\Sigma'\cdot\Sigma'\leq 2g(\Sigma')-2$, which implies that the map induced by the cobordism $X$ is zero \cite{km6}. It follows that $F_{K''}\equiv 0$. But this map sends $\theta(M_-,\Gamma_-,\xi_-)$ to $\theta(M,\Gamma,\xi)$, by Proposition \ref{prop:contactsurgeryrelativegirouxinstanton}. Thus, $\theta(M,\Gamma,\xi)=0.$
\end{proof}

\begin{remark} The idea above of using the right-handed trefoil was suggested to us by Peter Kronheimer and has been used to prove similar results; see \cite{mr}, for example.
\end{remark}

Given a closed 3-manifold $Y$, we denote by $Y(n)$ the sutured manifold obtained by removing $n$ disjoint 3-balls from $Y$, where the suture on each component of $\partial Y(n)$ consists of a single curve. The following is perhaps the most important result of this subsection.

\begin{theorem}
\label{thm:stein}
Suppose $(Y,\xi)$ is a closed contact manifold which is Stein fillable. Then the  invariant $\theta(Y(n),\xi|_{Y(n)})$ of the sutured contact manifold obtained from $(Y,\xi)$ by removing $n$ Darboux balls is nonzero.
\end{theorem}

As promised in the introduction, we have the following  corollary.

\begin{corollary}
\label{cor:embedding2} 
If $(M,\Gamma,\xi)$ embeds  as a sutured contact submanifold of a Stein fillable contact manifold, then $\theta(M,\Gamma,\xi) \neq 0$.
\end{corollary}

\begin{proof}
Suppose $(M,\Gamma,\xi)$ embeds in the Stein fillable contact manifold $(Y,\xi)$. Then $(M,\Gamma,\xi)$ also embeds into the complement $(Y(n),\xi|_{Y(n)})$ of some $n$ Darboux balls for any $n\geq 1$. By choosing these Darboux balls appropriately, we can arrange that $Y(n)\ssm\inr(M)$ has a contact handle decomposition consisting of $0$-, $1$-, and $2$-handles only. This corollary then follows from Theorem \ref{thm:stein} and the discussion in Remark \ref{rmk:embedding}.
\end{proof}

In order to prove Theorem \ref{thm:stein}, we first establish the following.

\begin{lemma}
\label{lem:s1s2}
For any $k\geq 0$ and any $n\geq 1$, $\SHItfun((\#^k(S^1\times S^2))(n))\cong \C^{2^{k+n-1}}$.
\end{lemma}

\begin{proof}
Note that  $(\#^k(S^1\times S^2))(n)$ can be obtained from the disjoint union of $k$ copies of $(S^1\times S^2)(1)$ with one copy of $S^3(n)$ via $k$ contact $1$-handle attachments. Note that each $(S^1\times S^2)(1)$ is obtained from $S^3(2)$ by attaching a single contact $1$-handle, and that $S^3(n)$ is obtained from the disjoint union of $n-1$ copies of $S^3(2)$ by attaching $n-2$ contact $1$-handles. Since contact $1$-handle attachment has no effect on the rank of sutured instanton homology, it follows that $\SHItfun((\#^k(S^1\times S^2))(n))$ is isomorphic to the sutured instanton homology of the disjoint union of $k+n-1$ copies of $S^3(2)$. In particular, \[\SHItfun((\#^k(S^1\times S^2))(n))\cong\bigotimes_{i=1}^{k+n-1} \SHItfun(S^3(2)).\] So, it suffices for the proof of this lemma to show that \begin{equation}\label{eqn:s3}\SHItfun(S^3(2))\cong \C^{2}.\end{equation}

Let $L_k$ denote the $k$-component unlink. Then $S^3(L_k)$ refers to the sutured manifold given as the complement of a regular neighborhood of $L_k$, with $2$ meridional sutures on each boundary component. Note that $S^3(L_k)$ can be obtained from $S^3(k)$ by attaching $k$ contact $1$-handles. Thus, \[\SHItfun(S^3(k))\cong \SHItfun(S^3(L_k)).\] The isomorphism class of the modules which make up the system $\SHItfun(S^3(L_k))$ is what Kronheimer and Mrowka call the instanton knot homology of $L_k$, denoted by $\KHI(L_k)$, so it suffices for \eqref{eqn:s3} to show that \[\KHI(L_2)\cong \C^2.\] In \cite{km5}, Kronheimer and Mrowka show that $\KHI$ satisfies an oriented skein exact triangle. Applying this to a diagram of $L_1$ with a single crossing, as in Figure \ref{fig:skeintri}, we have 
\[\xymatrix@C=-15pt@R=30pt{
\KHI(L_1) \ar[rr]^{f} & & \KHI(L_1) \ar[dl] \\
& \KHI(L_2). \ar[ul] & \\
} \]

\begin{figure}[ht]
\labellist
\small \hair 2pt
\pinlabel $L_1$ at 37 -10
\pinlabel $L_1$ at 132 -10
\pinlabel $L_2$ at 229 -10
\endlabellist
\centering
\includegraphics[width=6cm]{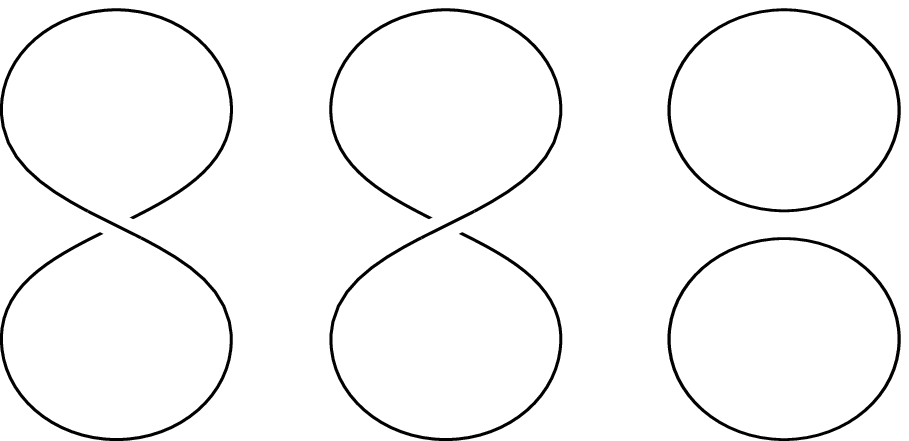}
\caption{The diagrams in an oriented skein triangle.}
\label{fig:skeintri}
\end{figure}

Since $S^3(1)$ is a product sutured manifold, we have that $\SHItfun(S^3(1))\cong \C$, which implies that $\KHI(L_1)\cong \C$. The map $f$ is therefore either zero or an isomorphism. If  the latter, then $\KHI(L_2)\cong 0$, which would imply that $\SHItfun(S^3(2))\cong 0$. But this is impossible since $S^3(2)$ is  taut \cite[Theorem 7.12]{km4}. Thus, $f\equiv 0$, which implies that $\KHI(L_2)\cong \C^2$.
\end{proof}

We may now prove Theorem \ref{thm:stein}.

\begin{proof}[Proof of Theorem \ref{thm:stein}]
Since $(Y,\xi)$ is Stein fillable, it is the result of contact $(-1)$-surgery on some link in the standard tight $(\#^k(S^1\times S^2),\xi_k)$. Let $\#^k(S^1\times S^2)(n)$ be the sutured contact manifold obtained by removing $n$ Darboux balls away from this link and let $Y(n)$ be the corresponding sutured contact manifold obtained via surgery (we are suppressing the contact structures from the notation). Then repeated application of Proposition \ref{prop:contactsurgeryrelativegirouxinstanton} gives rise to a map \[\SHItfun(-Y(n))\to\SHItfun( -(\#^k(S^1\times S^2)(n)))\] which sends $\theta(Y(n))$ to $\theta(\#^k(S^1\times S^2)(n))$. So, it suffices to show that \begin{equation}\label{eqn:xiknotzero}\theta(\#^k(S^1\times S^2)(n))\neq 0.\end{equation} 

The sutured contact manifold  $\#^k(S^1\times S^2)(n)$ has a partial open book decomposition given by $(S,P,id,\mathbf{c},f)$, where $S$ is obtained from the disk  $D^2$ by attaching $k$ unlinked $1$-handles $h_1,\dots,h_k$; $\mathbf{c} = \{c_1,\dots,c_{k+n-1}\}$, where $c_1,\dots,c_{k-1}$ are cocores of the $1$-handles $h_1,\dots,h_{k-1}$ and $c_k,\dots,c_{k+n-1}$ are parallel cocores of the $1$-handle $h_k$; and $P$ is a regular neighborhood of these cocores, as shown in Figure \ref{fig:openbooks1s2}. Define $M_0=H(S)$ and let $M_i$ be the sutured contact manifold obtained by attaching contact $2$-handles to $H(S)$ along the curves $\gamma_1,\dots,\gamma_i\subset \boldsymbol{\gamma}(h,\mathbf{c})$ for $i\geq 1$. In particular, $M_i$ is obtained from $M_{i-1}$ by attaching a contact $2$-handle  along $\gamma_i\subset \partial M_{i-1}$. Let \[\mathscr{H}_{\gamma_i}:\SHItfun(-M_{i-1})\to\SHItfun(-M_i)\] denote the corresponding morphism. Note that $M_{k+n-1} = M(S,P,id,\mathbf{c})$ and the contact invariant $\theta(\#^k(S^1\times S^2)(n))$ is the image of \begin{equation}\label{eqn:nonzeroone}(\mathscr{H}_{\gamma_{k+n-1}}\circ\dots\circ\mathscr{H}_{\gamma_1})(\mathbf{1})\end{equation} under the map $\SHItfun(f)$. So, to prove \eqref{eqn:xiknotzero}, it suffices to show that the class in \eqref{eqn:nonzeroone} is nonzero. For this, it suffices to show that each $\mathscr{H}_{\gamma_i}$ is injective.

\begin{figure}[ht]
\labellist
\small \hair 2pt
\pinlabel $c_1$ at 98 -7
\pinlabel $c_2$ at 0 46
\pinlabel \rotatebox{45}{$c_{k-1}$} at 11 186
\pinlabel \rotatebox{47}{$c_{k}$} at 88 215
\pinlabel \rotatebox{-50}{$c_{r}$} at 147 199
\endlabellist
\centering
\includegraphics[width=4cm]{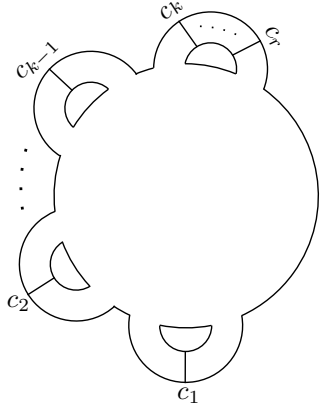}
\caption{The page $S$ obtained from the disk by attaching $k$ unlinked $1$-handles showing the cocores $c_1,\dots,c_{k-1}$ on the first $k-1$ handles and the cocores $c_k,\dots,c_{r}$ on the last handle, where $r=k+n-1$. }
\label{fig:openbooks1s2}
\end{figure}

 Let $\data_{i-1} = (Y_{i-1},R,r,m_{i-1},\eta,\alpha)$ be a marked odd closure of $M_{i-1}$ and let $\data_i = (Y_i,R,r,m_i,\eta,\alpha)$ be the induced  closure of $M_i$, where $Y_i$ is  obtained from $Y_{i-1}$ by performing $(\partial M_{i-1})$-framed surgery on $m(\gamma_i')$, where $\gamma_i'$ is a pushoff of $\gamma_i$ into the interior of $M_{i-1}$. The  $2$-handle cobordism associated to this surgery gives rise to a map
 \[g:\SHIt(-\data_{i-1})\to\SHIt(-\data_i)\] whose equivalence class agrees with the map $\mathscr{H}_{\gamma_i}$. Note that $\gamma_i$ is a unknot in $M_{i-1}$ such that the framing induced by the bounding disk agrees with the $(\partial M_{i-1})$-framing. In other words, $Y_i$ is obtained from $Y_{i-1}$ via $0$-surgery on the unknot $m(\gamma_i')$. Let $\data_{i-1}' = (Y_{i-1}',R,r,m_{i-1}',\eta,\alpha)$ be the  closure of $M_{i-1}$ in which $Y_{i-1}'$ is  obtained from $Y_{i-1}$ by $(-1)$-surgery on this unknot. Then $g$ fits into the surgery exact triangle 
 \begin{equation} \label{eqn:exacttrishi}\xymatrix@C=-25pt@R=30pt{
\SHIt(-\data_{i-1}) \ar[rr]^{g} & & \SHIt(-\data_i) \ar[dl] \\
& \SHIt(-\data_{i-1}'). \ar[ul] & \\
} \end{equation}
For $i=1,\dots,k$, let $(S_i,P_i,id,\mathbf{c}_i = \{c_1,\dots,c_i\})$ be the partial open book in which $S_i$ is the surface obtained from the disk by attaching the first $i$ $1$-handles $h_1,\dots,h_i$. Then $M(S_i,P_i,id,\mathbf{c}_i)$ is diffeomorphic to $(\#^i(S^1\times S^2))(1)$. Note that $M_i$ is obtained from $M(S_i,P_i,id,\mathbf{c}_i)$ by attaching contact $1$-handles. Therefore, \[\SHItfun(-M_i) \cong \SHItfun(-(\#^i(S^1\times S^2))(1))\cong \C^{2^i},\]  where the latter isomorphism is by Lemma \ref{lem:s1s2}. It follows that \[\SHIt(-\data'_{i-1})\cong \C^{2^{i-1}}\,\,\,\,\,{\rm and }\,\,\,\,\, \SHIt(-\data_{i-1})\cong \C^{2^{i-1}}\,\,\,\,\,{\rm and}\,\,\,\,\,\SHIt(-\data_i)\cong \C^{2^i}.\] The exactness of the triangle in \eqref{eqn:exacttrishi} then implies that $g$ is injective for  $i=1,\dots,k$. 

For $i=k+1,\dots, k+n-1$, $M_i$ is diffeomorphic to $(\#^k(S^1\times S^2))(1+i-k)$. Therefore, \[\SHItfun(-M_i) \cong \SHItfun(-(\#^k(S^1\times S^2))(1+i-k))\cong \C^{2^{i}}\] in this case as well,  by Lemma \ref{lem:s1s2}. We therefore have again that \[\SHIt(-\data'_{i-1})\cong \C^{2^{i-1}}\,\,\,\,\,{\rm and }\,\,\,\,\, \SHIt(-\data_{i-1})\cong \C^{2^{i-1}}\,\,\,\,\,{\rm and}\,\,\,\,\,\SHIt(-\data_i)\cong \C^{2^i}.\] The exactness of the triangle in \eqref{eqn:exacttrishi} then implies that $g$ is injective for  $i=k+1,\dots,k+n-1$.

Putting all of this together, we have shown that $\mathscr{H}_{\gamma_i}$ is injective for all $i=1,\dots,k+n-1$, completing the proof.
\end{proof}

\section{Stein fillings and the fundamental group}
\label{sec:stein}

Below, we demonstrate how Conjecture \ref{conj:steinpi1}  follows from Conjecture \ref{conj:steinrank}. Suppose $Y$  is an integer homology 3-sphere which bounds a Stein 4-manifold $(X,J)$ with $c_1(J)\neq 0$. The long exact sequence of the pair $(X,Y)$, combined with Poincar{\'e} duality, tells us that \[H^2(X)\cong H_2(X,Y) \cong H_2(X).\] Moreover, $H_2(X)$ is nontorsion since $X$ can be built out of $1$- and $2$-handles. Thus, $H^2(X)$ is nontorsion. In particular, the difference between two unequal elements in $H^2(X)$ is nontorsion. Let $\bar J$ be the conjugate Stein structure on $X$, so that $c_1(\bar J) = -c_1(J)$. It then follows from the discussion above that $c_1(J)\neq c_1(\bar J)$ and, hence, that   $c_1(J)-c_1(\bar J)$ is nontorsion. 
Assuming that Conjecture \ref{conj:steinrank} is true, it   follows that the rank of $\SHItfun(-Y(1))$ is at least $2$. But \[rk(\SHItfun(-Y(1)))=rk(\SHItfun(-Y(U))),\] where $U$ is an unknot in $Y$. Therefore, \begin{equation}\label{eqn:rk2}rk(\KHIt(Y,U))\geq 2.\end{equation} 

We claim that there exists an irreducible homomorphism \begin{equation}\label{eqn:rhoknot}\rho:\pi_1(Y\ssm U)\to SU(2)\end{equation} which sends a chosen meridian $m$ of $U$ to  $\mathbf{i}\subset SU(2)$. The argument is similar to that used in the proof of \cite[Proposition 7.17]{km4}. Suppose there are no irreducibles. Observe that there is only one reducible homomorphism.  Indeed, reducibles have abelian image and so must  factor through homomorphisms \[H_1(Y\ssm U;\mathbb{Z})\to SU(2)\] sending $[m]$ to $\mathbf{i}$. But  $H_1(Y\ssm U;\mathbb{Z})\cong\Z$ since $Y$ is an integer homology $3$-sphere, so there is exactly one such homomorphism. 
Since there are no irreducibles, this  reducible  homomorphism corresponds to the unique generator of a chain complex for the reduced singular instanton knot homology $I^{\natural}(Y,U)$, which implies that \[I^{\natural}(Y,U)\cong \Z.\] There are several ways to see this; it follows easily, for instance, from the work of Hedden, Herald, and Kirk \cite{hhk}. On the other hand, Kronheimer and Mrowka proved in \cite{km3} that \[\KHIt(Y,U)\cong I^{\natural}(Y,U)\otimes \C,\] so the inequality in (\ref{eqn:rk2}) implies that $I^{\natural}(Y,U)$ has rank at least 2, a contradiction. It follows that there exists an irreducible homomorphism \[\rho: \pi_1(Y\ssm U)\cong \pi_1(Y)*\Z\to SU(2)\] as claimed. Such a $\rho$ then induces a  homomorphism \[\rho_Y:\pi_1(Y)\to SU(2)\] which must be nontrivial (otherwise, $\rho$ would be reducible), completing this discussion. 

\bibliographystyle{hplain}
\bibliography{References}

\end{document}